\pdfoutput=1  

\documentclass[12pt]{amsart}
\usepackage[utf8]{inputenc}

\usepackage[colorlinks]{hyperref}
\hypersetup{
    breaklinks=true,
    colorlinks=true,
    linkcolor=dodger,
    filecolor=magenta,
    urlcolor=dodger,
    citecolor=dodger,
}

\usepackage[a4paper, top=3cm, left=2.5cm, right=2.5cm, bottom=3cm]{geometry}

\usepackage{graphicx}
\usepackage{amsmath,amssymb,enumitem,mathrsfs}
\usepackage{color,amscd}
\usepackage{tikz}
\usepackage{float}
\usepackage{multicol,multirow}
\usetikzlibrary{arrows, intersections, cd}
\tikzset{
  shift left/.style ={commutative diagrams/shift left={#1}},
  shift right/.style={commutative diagrams/shift right={#1}}
}

\pgfdeclarelayer{bg}
\pgfdeclarelayer{fg}
\pgfdeclarelayer{crossing over}
\pgfsetlayers{main,bg,crossing over,fg}

\usepackage{thmtools,cleveref}

\usepackage{lineno}

\usepackage{comment}

\setlist{itemsep=1ex, topsep=0ex}

\definecolor{brown}{rgb}{0.43, 0.21, 0.1}
\definecolor{green}{cmyk}{0.64,0,0.95,0.40}
\definecolor{dodger}{rgb}{0.0,0.5,1.0}
\definecolor{ballblue}{rgb}{0.13, 0.67, 0.8}

\newcommand{\vs}[1]{{\color{blue}VG: #1}}

\newcommand{\dm}[1]{{\color{green}DM: #1}}

\def\ml{l\kern-0.035cm\char39\kern-0.03cm}

\newcounter{enuAlph}

\theoremstyle{plain}
\newtheorem{theorem}{Theorem}[section]
\newtheorem{lemma}[theorem]{Lemma}
\newtheorem{mainlemma}[theorem]{Main Lemma}
\newtheorem{fact}[theorem]{Fact}

\newtheorem{corollary}[theorem]{Corollary}

\newtheorem{mainthm}[enuAlph]{Theorem}

\theoremstyle{definition}
  \newtheorem{definition}[theorem]{Definition}
  \newtheorem{example}[theorem]{Example}
  \newtheorem{remark}[theorem]{Remark}
  
  \newtheorem{question}[theorem]{Question}

  \numberwithin{equation}{theorem}

\newcommand{\N}{{\mathbb N}}
\newcommand{\Z}{{\mathbb Z}}
\newcommand{\Q}{{\mathbb Q}}
\newcommand{\R}{{\mathbb R}}
\newcommand{\meager}{\mathcal{M}}
\newcommand{\Nideal}{\mathcal{N}}
\newcommand{\NidealJ}[1]{\mathcal{N}_{#1}}
\newcommand{\NstaridealJ}[1]{\mathcal{N}^*_{#1}}
\newcommand{\fin}{\mathrm{Fin}}
\newcommand{\Fin}{\mathrm{Fin}}
\newcommand{\seqn}[2]{\left\langle #1 :\, #2\right\rangle}
\newcommand{\set}[2]{\left\{ #1 :\, #2 \right\}}
\newcommand{\leb}[1]{\mu\left(#1\right)}
\newcommand{\leqRB}{\mathrel{\leq_{\mathrm{RB}}}}
\newcommand{\leqKB}{\mathrel{\leq_{\mathrm{KB}}}}
\newcommand{\leqKBpr}{\mathrel{\leq_{\overline{\mathrm{KB}}}}}
\newcommand{\cantor}{{}^\omega 2}
\newcommand{\baire}{{}^\omega \omega}

\DeclareMathOperator{\add}{\mathrm{add}}
\DeclareMathOperator{\cov}{\mathrm{cov}}
\DeclareMathOperator{\non}{\mathrm{non}}
\DeclareMathOperator{\cof}{\mathrm{cof}}
\DeclareMathOperator{\cf}{\mathrm{cf}}
\DeclareMathOperator{\cl}{\mathrm{cl}}
\DeclareMathOperator{\sqw}{\mathrm{sq}_{<\omega}}

\newcommand{\Lc}{\mathbf{Lc}}
\newcommand{\Rbf}{\mathbf{R}}
\newcommand{\Sbf}{\mathbf{S}}
\newcommand{\cbf}{\mathbf{c}}
\newcommand{\Cbf}{\mathbf{C}}

\newcommand{\Acal}{A}
\newcommand{\Ecal}{\mathcal{E}}
\newcommand{\Fcal}{F}
\newcommand{\Jcal}{J}
\newcommand{\Kcal}{K}
\newcommand{\Mcal}{\mathcal{M}}
\newcommand{\Ncal}{\mathcal{N}}

\newcommand{\SNcal}{\mathcal{SN}}
\newcommand{\Scal}{S}

\newcommand{\Ucal}{U}

\newcommand{\Icl}{\mathcal{I}}
\newcommand{\Jcl}{\mathcal{J}}
\newcommand{\Kcl}{\mathcal{K}}

\DeclareMathOperator{\pts}{\mathcal{P}}
\DeclareMathOperator{\dom}{\mathrm{dom}}

\newcommand{\Cbb}{\mathbb{C}}
\newcommand{\Pbb}{\mathbb{P}}

\newcommand{\Qnm}{\dot{\mathbb{Q}}}

\newcommand{\afrak}{\mathfrak{a}}
\newcommand{\bfrak}{\mathfrak{b}}

\newcommand{\blc}{\mathfrak{b}^{\rm Lc}}
\newcommand{\cfrak}{\mathfrak{c}}
\newcommand{\dfrak}{\mathfrak{d}}
\newcommand{\efrak}{\mathfrak{e}}

\newcommand{\dlc}{\mathfrak{d}^{\rm Lc}}
\newcommand{\gfrak}{\mathfrak{g}}
\newcommand{\hfrak}{\mathfrak{h}}
\newcommand{\ifrak}{\mathfrak{i}}
\newcommand{\sfrak}{\mathfrak{s}}
\newcommand{\rfrak}{\mathfrak{r}}

\newcommand{\pfrak}{\mathfrak{p}}

\newcommand{\ufrak}{\mathfrak{u}}

\newcommand{\la}{\langle}
\newcommand{\ra}{\rangle}
\newcommand{\leqT}{\mathrel{\preceq_{\mathrm{T}}}}
\newcommand{\eqT}{\mathrel{\cong_{\mathrm{T}}}}

\newcommand{\frestr}{{\upharpoonright}}

\newcommand{\mnula}{{\mathcal{N}}}
\newcommand{\lone}[1]{{\ell^1({#1})}}

\newcommand{\omegabar}{{\overline{\Omega}}}

\newcommand{\con}{{\mathfrak c}}
\newcommand{\be}{\mathfrak{b}}
\newcommand{\de}{\mathfrak{d}}

    \newcommand{\imp}{\mathrel{\mbox{$\Rightarrow$}}}

    \newcommand{\sii}{\mathrel{\mbox{$\Leftrightarrow$}}}

\input cyracc.def

\begin{document}

\title{Lebesgue measure zero modulo ideals on the natural numbers}

\author[V.\ Gavalová]{Viera Gavalová}
\address{Department of Applied Mathematics and Business Informatics\\ Faculty of Economics\\ Technical University of Ko\v{s}ice\\ N\v emcovej 32\\ 040 01 Ko\v{s}ice\\ Slovakia}
\email{viera.gavalova@tuke.sk}
\urladdr{http://www.researchgate.net/profile/Viera-Gavalova}

\author[D.\,A.\ Mej\'{i}a]{Diego A.\ Mej\'{i}a}
\address{
Creative Science Course (Mathematics)\\
Faculty of Science\\
Shizuoka University\\
836 Ohya, Suruga-ku, Shizuoka City\\
Shi\-zu\-o\-ka Prefecture, 422-8529\\
Japan
}
\email{diego.mejia@shizuoka.ac.jp}
\urladdr{http://www.researchgate.net/profile/Diego\_Mejia2}

\date{\today}

\begin{abstract}
    We propose a reformulation of the ideal $\mathcal{N}$ of Lebesgue measure zero sets of reals modulo an ideal $J$ on $\omega$, which we denote by $\mathcal{N}_J$. In the same way, we reformulate the ideal $\mathcal{E}$ generated by $F_\sigma$ measure zero sets of reals modulo $J$, which we denote by $\mathcal{N}^*_J$. We show that these are $\sigma$-ideals and that $\mathcal{N}_J=\mathcal{N}$ iff $J$ has the Baire property, which in turn is equivalent to $\mathcal{N}^*_J=\mathcal{E}$. Moreover, we prove that $\mathcal{N}_J$ does not contain co-meager sets and $\mathcal{N}^*_J$ contains non-meager sets when $J$ does not have the Baire property. We also prove a deep connection between these ideals modulo $J$ and the notion of \emph{nearly coherence of filters} (or ideals).
    
    We also study the cardinal characteristics associated with $\mathcal{N}_J$ and $\mathcal{N}^*_J$. We show their position with respect to Cicho\'n's diagram and prove consistency results in connection with other very classical cardinal characteristics of the continuum, leaving just very few open questions. To achieve this, we discovered a new characterization of $\mathrm{add}(\mathcal{N})$ and $\mathrm{cof}(\mathcal{N})$. We also show that, in Cohen model, we can obtain many different values to the cardinal characteristics associated with our new ideals.
\end{abstract}



\thanks{The first author was supported by APVV-20-0045 of 
the Slovak Research and Development Agency, and the second author was supported by the Grant-in-Aid for Early Career Scientists 18K13448 and the Grant-in-Aid for Scientific Research (C)  23K03198, Japan Society for the Promotion of Science.}

\maketitle

\section{Introduction}\label{sec:pre}

Many notions of topology and combinatorics of the reals have been reformulated and investigated in terms of ideals on the natural numbers (always assuming that an ideal contains all the finite sets of natural numbers). For instance, the usual notion of convergence on a topological space, which states that a sequence $\la x_n:\, n<\omega\ra$ in a topological space converges to a point $x\in X$ when the set $\set{n<\omega}{x_n\notin U}$ is finite for any open neighborhood $U$ of $x$, is generalized in terms of ideals $J$ on the natural numbers by changing the latter requirement by $\set{n<\omega}{ x_n\notin U}\in J$ (see e.g.~\cite{I-conv}). More recent and remarkable examples are the so called \emph{selection principles}, which are reformulated in terms of ideals, and show deep connections with cardinal characteristics of the real line~\cite{BDS,SJ,SS,Miro1,Miro2}.

In combinatorics of the real line,
some classical cardinal characteristics have been reformulated in terms of ideals (and in many cases they are connected to selection principles in topology). The most natural examples are the reformulations of the bounding number $\bfrak_J$ and the dominating number $\dfrak_J$ in terms of an ideal $J$ on $\omega$, more concretely, with respect to the relation $\leq ^J$ on $\baire$, which states that $x \leq^J y$ iff $\set{n<\omega}{x(n)\nleq y(n)}\in J$. These have been investigated by e.g.\ Canjar~\cite{Canjar}, Blass and Mildenberger~\cite{BM}, also in connection with arithmetic in the sense that, for any maximal ideal $J$, $\bfrak_J=\dfrak_J$ is the cofinality of the ultrapower (on the dual filter of $J$) of $\omega$. Other classical cardinal characteristics have been reformulated in terms of ideals on $\omega$, like the almost disjointness number~\cite{FaSo,Bakke,RaSt} and the pseudo-intersection number~\cite{BNF,SJpseudo}, among others (see \cite[Sec.~8.4]{Hrusak}).

In the~present paper we offer a reformulation, in terms of ideals on $\omega$, of the ideal of Lebesgue measure zero subsets of the reals. Our reformulation does not come from a definition of the Lebesgue measure in terms of an ideal $J$ on $\omega$, but it is inspired from one combinatorial characterization of measure zero. Details of this new definition are provided in \autoref{sec:NmodI}.
We work in the Cantor space $\cantor$ for simplicity, but the same reformulations and results can be obtained in other standard Polish spaces with a measure 
(\autoref{sec:var} presents a detailed discussion). We denote by $\Ncal_J$ the collection of \emph{null subsets modulo $J$} of $\cantor$. It will be clear that $\Ncal_\Fin=\Ncal$, the ideal of Lebesgue measure subsets of $\cantor$, where $\Fin$ denotes the ideal of finite subsets of $\omega$.

We also provide a reformulation of $\Ecal$, the ideal generated by the $F_\sigma$ measure zero subsets of $\cantor$, in terms of an ideal $J$ on $\omega$, which we denote by $\Ncal^*_J$. As expected, we have $\Ncal^*_\Fin = \Ecal$.

We obtain that, for any ideal $J$ on the natural numbers, $\Ncal_J$ and $\Ncal^*_J$ are actually $\sigma$-ideals on $\cantor$ and that, whenever $K$ is another ideal on $\omega$ and $J\subseteq K$,
\[\Ecal\subseteq \Ncal^*_J\subseteq \Ncal^*_K\subseteq \Ncal_K\subseteq \Ncal_J\subseteq \Ncal.\]
In fact, it will be clear from the definitions that, whenever $K$ is a maximal ideal on $\omega$, $\Ncal^*_K=\Ncal_K$.

Our first set of main results work as interesting characterizations of ideals on $\omega$ with the Baire property:

\begin{mainthm}\label{mainBP}
Let $J$ be an ideal on $\omega$. Then, the following statements are equivalent:
\begin{enumerate}[label = \rm (\roman*)]
    \item $J$ has the Baire property.
    \item $\Ncal_J=\Ncal$.
    \item $\Ncal^*_J=\Ecal$.
\end{enumerate}
\end{mainthm}

Although no new ideals on the reals are obtained from ideals with the Baire property, we obtain new characterizations of the ideals $\Ncal$ and $\Ecal$. Moreover, ideals without the Baire property offer new ideals on the reals that are worth of research: the previous result can be expanded in connection with $\Mcal$, the ideal of meager subsets of $\cantor$.

\begin{mainthm}\label{mainBP2}
Let $J$ be an ideal on $\omega$. Then, the following statements are equivalent:
\vspace{-5pt}
\begin{multicols}{2}
\begin{enumerate}[label = \rm (\roman*)]
    \item $J$ does not have the Baire property.
    \item $\Ncal_J\subsetneq\Ncal$.
    \item $\Ecal \subsetneq \Ncal^*_J$.
    \item No member of $\Ncal_J$ is co-meager.
    \item $\Mcal\cap\Ncal\nsubseteq\Ncal_J$.
    \item $\Ncal^*_J\nsubseteq\Mcal$.
\end{enumerate}
\end{multicols}
\end{mainthm}

\autoref{mainBP} and~\ref{mainBP2} summarize \autoref{BaireForNJ2}, \autoref{cor:notcomeager}, \autoref{charaktNj2} and~\ref{conseqforE2}.
There are two elements providing the proof of these results. The first corresponds to monotonicity results with respect to the well-known \emph{Kat\v etov-Blass order $\leqKB$} and \emph{Rudin-Blass order $\leqRB$} between ideals (\autoref{RBeq}), 
and the second is \emph{Bartoszy\'nski's and Scheepers'} game ~\cite{BSch} that characterizes filters (and hence ideals) with the Baire property, which we use to prove many properties of $\Ncal_J$ and $\Ncal^*_J$ for any ideal $J$ on $\omega$ without the Baire property, specifically that $\Ncal_J$ cannot contain co-meager subsets of $\cantor$, and that $\Ncal^*_J$ contains non-meager sets.

About the connection between $\Ncal_J$ and $\Ncal_K$ for different ideals $J$ and $K$ on $\omega$, (and likewise for $\Ncal^*_J$ and $\Ncal^*_K$), we discovered a deep connection between these ideals and the notion of \emph{nearly coherence of ideals (or filters)} on $\omega$, original from Blass~\cite{blass86}. The ideals $J$ and $K$ are \emph{nearly coherent} if there is some finite-to-one function $f\colon \omega\to \omega$ such that $\set{y\subseteq\omega}{f^{-1}[y]\in J\cup K}$ generates an ideal. We prove that nearly coherence of ideals is characterized as follows:

\begin{mainthm}[\autoref{nc-char}]\label{maincoherence}
Let $J$ and $K$ be ideals on $\omega$. Then the following statements are equivalent:
\begin{enumerate}[label = \rm (\roman*)]
    \item\label{it:coh1} $J$ and $K$ are nearly coherent.
    \item\label{it:coh2} There is some ideal $K'$ such that $\Ncal^*_J\cup \Ncal^*_K\subseteq \Ncal^*_{K'}\subseteq \Ncal_{K'}\subseteq \Ncal_J\cap \Ncal_K$.
\item $\Ncal^*_J\subseteq\Ncal_K$.
\end{enumerate}
\end{mainthm}

This means that, whenever $J$ and $K$ are \underline{not} nearly coherent, the ideals $\Ncal_J$ and $\Ncal_K$ are quite different, likewise for $\Ncal^*_J$ and $\Ncal^*_K$.

Blass and Shelah~\cite{blassshelah} proved that it is consistent with ZFC that any pair of ideals are nearly coherent, which is known as \ref{NCF}, the principle of \emph{nearly coherence of filters}. \autoref{maincoherence} 
implies that, under \ref{NCF}, there is only one $\Ncal_J$ for maximal ideals $J$ on $\omega$. We still do not know whether \ref{NCF} implies that there is just one $\Ncal_J$ (or $\Ncal^*_J$) for $J$ without the Baire property. On the other hand, when we assume that there are not nearly coherent ideals (which is consistent with ZFC, e.g.\ it is valid under CH and in random model, see~\cite[Sec.~4]{blass86}), we can construct a non-meager ideal $K$ on $\omega$ such that $\Ncal^*_K\neq \Ncal_K$ (\autoref{noncoh:sum}).  
In contrast with the previous question, we do not know whether ZFC proves the existence of an ideal $K$ without the Baire property such that $\Ncal^*_K\neq \Ncal_K$.

The proof of \autoref{maincoherence} uses Eisworth's game that characterizes nearly coherence~\cite{Eis01}. Another element relevant to this proof is the order $\leqKBpr$, which is the dual of the Kat\v{e}tov-Blass order (see \autoref{def:RB}). 
If $J$ and $K$ are nearly coherent then it is clear that there is some ideal $K'$ such that $J,K\leqKBpr K'$, 
but the converse is also true thanks to \autoref{maincoherence}. This equivalence is claimed in~\cite{blass86}, but here we present an alternative proof using our new ideals.

We also study the cardinal characteristics associated with the ideals $\Ncal_J$ and $\Ncal^*_J$, i.e., additivity, covering, uniformity, and cofinality. Recall that $\sfrak$ denotes the \emph{splitting number} and $\rfrak$ the \emph{reaping number}.\footnote{We assume that the reader is somewhat familiar with classical cardinal characteristics of the continuum, so we do not repeat their definitions in this paper. The reader can refer to e.g.~\cite{blassbook}.} In ZFC, we can prove:

\begin{mainthm}\label{mainZFCcard}
Let $J$ be an ideal on $\omega$. With respect to Cicho\'n's diagram (see \autoref{fig:cichon}):
\begin{enumerate}[label =\rm (\alph*)]
    \item\label{it:covnon} $\cov(\Ncal)\leq \cov(\Ncal_J) \leq \cov(\Ncal^*_J)\leq \cov(\Ecal) \leq \min\{\cof(\Mcal),\rfrak\}$ and \[\max\{\add(\Mcal),\sfrak\}\leq \non(\Ecal)\leq \non(\Ncal^*_J)\leq \non(\Ncal_J) \leq \non(\Ncal),\] 
    i.e.\ the coverings of $\Ncal_J$ and $\Ncal^*_J$ are between $\cov(\Ncal)$ and $\min\{\cof(\Mcal),\rfrak\}$, 
    and their uniformities are between $\min\{\add(\Mcal),\sfrak\}$ and $\non(\Ncal)$.
    
    \item\label{it:addcof} The additivites of $\Ncal_J$ and $\Ncal^*_J$ are between $\add(\Ncal)$ and $\cov(\Mcal)$, 
    and their cofinalities are between $\non(\Mcal)$ and $\cof(\Ncal)$.
\end{enumerate}
\end{mainthm}

\begin{figure}[ht]
\centering
\begin{tikzpicture}
\small{
 \node (aleph1) at (-1.5,0) {$\aleph_1$};
 \node (addn) at (0,0){$\add(\Ncal)$};
 \node (adNJ) at (1,2) {$\add(\Ncal_J)$};
 \node (adNJ*) at (3,3) {$\add(\Ncal^*_J)$};
 \node (covn) at (0,10){$\cov(\Ncal)$};
 \node (cvNJ) at (1,8) {$\cov(\Ncal_J)$};
 \node (cvNJ*) at (3,7) {$\cov(\Ncal^*_J)$};
 \node (cove) at (5,6) {$\cov(\Ecal)$};
 \node (b) at (4,5) {$\bfrak$};
 \node (addm) at (4,0) {$\add(\Mcal)$} ;
 \node (nonm) at (4,10) {$\non(\Mcal)$} ;
 \node (none) at (7,4) {$\non(\Ecal)$};
 \node (d) at (8,5) {$\dfrak$};
 \node (covm) at (8,0) {$\cov(\Mcal)$} ;
 \node (cfm) at (8,10) {$\cof(\Mcal)$} ;
 \node (nonn) at (12,0) {$\non(\Ncal)$} ;
 \node (nnNJ) at (11,2) {$\non(\Ncal_J)$};
 \node (nnNJ*) at (9,3) {$\non(\Ncal^*_J)$};
 \node (cfn) at (12,10) {$\cof(\Ncal)$} ;
 \node (cfNJ) at (11,8) {$\cof(\Ncal_J)$};
 \node (cfNJ*) at  (9,7){$\cof(\Ncal^*_J)$};
 \node (s) at (-1,3) {$\sfrak$};
 \node (r) at (13,7) {$\rfrak$};
 \node (c) at (13.5,10) {$\cfrak$};

\foreach \from/\to in {
aleph1/addn, cfn/c, aleph1/s, r/c,
s/none, s/d, b/r, cove/r,
adNJ/nnNJ, adNJ*/nnNJ*, cvNJ/cfNJ, cvNJ*/cfNJ*,
addm/b, b/nonm, b/d, addm/covm, covm/d, nonm/cfm, d/cfm,
addm/none, cove/cfm,
covm/cove, none/nonm,
addn/adNJ, addn/adNJ*,
covn/cvNJ, cvNJ/cvNJ*, cvNJ*/cove,
addn/covn, adNJ/cvNJ, adNJ*/cvNJ*,
addn/addm, covn/nonm, covm/nonn, cfm/cfn, nonn/cfn,
adNJ/covm, adNJ*/covm,
none/nnNJ*, nnNJ*/nnNJ, nnNJ/nonn,
nnNJ/cfNJ, nnNJ*/cfNJ*,
cfNJ*/cfn, cfNJ/cfn,
nonm/cfNJ*, nonm/cfNJ}
{
\path[-,draw=white,line width=3pt] (\from) edge (\to);
\path[->,] (\from) edge (\to);
}
}
\end{tikzpicture}
%
%
\caption{Cicho\'n's diagram including the cardinal characteristics associated with our ideals, $\sfrak$ and $\rfrak$, as stated in \autoref{mainZFCcard}.}
\label{fig:cichon}
\end{figure}

The previous theorem summarizes \autoref{covnonvscichon} and \autoref{addN-covM}.
Item~\ref{it:covnon} follows directly by the subset relation between the ideals, and also because $\add(\Ecal)=\add(\Mcal)$ and $\cof(\Ecal)=\cof(\Mcal)$ due to Bartoszy\'nski and Shelah~\cite{BartSh}. Results from the latter reference guarantee easily the connections with $\cov(\Mcal)$ and $\non(\Mcal)$ in~\ref{it:addcof}, but the connections with $\add(\Ncal)$ and $\cof(\Ncal)$ require quite some work. To prove this, we define two cardinal characteristics $\bfrak_J(\omegabar)$ and  $\dfrak_J(\omegabar)$. It will not be hard to show that the additivites and cofinalities of $\Ncal_J$ and $\Ncal^*_J$ are between $\bfrak_J(\omegabar)$ and $\dfrak_J(\omegabar)$ (\autoref{def:Sbf}), and that $\bfrak_\Fin(\omegabar)\leq \bfrak_J(\omegabar)$ and $\dfrak_J(\omegabar)\leq \dfrak_\Fin(\omegabar)$. The real effort is to prove the following new characterization of $\add(\Ncal)$ and $\cof(\Ncal)$.

\begin{mainthm}[\autoref{omegabarFin}]\label{maincharadd}
$\bfrak_\Fin(\omegabar)=\add(\Ncal)$ and $\dfrak_\Fin(\omegabar)=\cof(\Ncal)$.
\end{mainthm}

In terms of inequalities with classical characteristics of the continuum, \autoref{mainZFCcard} seems to be the most optimal: we also manage to prove that, in most cases, no further inequalities can be proved, not just with the cardinals in Cicho\'n's diagram, but with many classical cardinal characteristics as illustrated in \autoref{fig:all20}. We just leave few open questions, for example, whether it is consistent that $\cov(\Ncal_J)<\add(\Mcal)$ (and even smaller than the pseudo-intersection number $\pfrak$) for some maximal ideal $J$ (likewise for $\cof(\Mcal)<\non(\Ncal_J)$). This is all dealt with in \autoref{sec:cons}.

Many consistency results supporting the above comes from the forcing model after adding uncountably many Cohen reals.

\begin{mainthm}[\autoref{cohenthm}]\label{mainCohen}
Let $\lambda$ be an uncountable cardinal. After adding $\lambda$-many Cohen reals: for any regular uncountable $\kappa\leq\lambda$ there is some (maximal) ideal $J^\kappa$ on $\omega$ such that $\add(\Ncal_{J^\kappa})=\cof(\Ncal_{J^\kappa})=\kappa$.
\end{mainthm}

This shows that there are many different values for the cardinal characteristics associated with different $\Ncal_J$ after adding many Cohen reals. This potentially shows that many of these values can be strictly between $\non(\Mcal)$ and $\cov(\Mcal)$ because, after adding $\lambda$-many Cohen reals, $\non(\Mcal)=\aleph_1$ and $\lambda\leq\cov(\Mcal)$ (and $\cfrak=\lambda$ when $\lambda^{\aleph_0}=\lambda$). See more in \autoref{sec:cons}, specifically item~\ref{modelCohen}. This is inspired in Canjar's result~\cite{Canjar} stating that, after adding $\lambda$ many Cohen reals, for any uncountable regular $\kappa\leq\lambda$ there is some (maximal) ideal $J^\kappa$ such that $\bfrak_{J^\kappa}=\dfrak_{J^\kappa}=\kappa$.

\subsection*{Structure of the paper}
In \autoref{sec:NmodI} we define $\Ncal_J$ and $\Ncal^*_J$, prove their basic properties, the monotonicity with respect to the orders $\leqKB$, $\leqKBpr$ and $\leqRB$, and that $\Ncal_J=\Ncal$ and $\Ncal^*_J=\Ecal$ when $J$ has the Baire property. In \autoref{sec:nBP} we deal with ideals without the Baire property, finish the proof of \autoref{mainBP}, and prove \autoref{mainBP2}. \autoref{sec:near} is devoted to our results related to nearly coherence of ideals, specifically with the proof of \autoref{maincoherence}. \autoref{sec:cardinalityNJ} presents ZFC results about the cardinal characteristics associated with our new ideals, mainly the proof of \autoref{mainZFCcard} and \autoref{maincharadd}, and \autoref{sec:cons} deals with the consistency results and \autoref{mainCohen}.

In addition, we discuss in \autoref{sec:var} the notions of $\Ncal_J$ and $\Ncal^*_J$ in other classical Polish spaces (like $\R$) and some weak version of our ideals. Finally, in \autoref{sec:Q} we summarize some open questions related to this work.

\subsection*{Acknowledgments} The authors thank the~participants of Ko\v sice set-theoretical seminar, especially Jaroslav \v Supina and Miguel Cardona, for very fruitful discussions about this work, and the anonymous referee for the constructive comments.


\section{Measure zero modulo ideals}\label{sec:NmodI}

We first present some basic notation. In general, by an~\emph{ideal on $M$} we understand a~family $\Jcal\subseteq \pts(M)$ that is hereditary (i.e. $a\in \Jcal$ for any $a \subseteq b\in \Jcal$), closed under finite unions, containing all finite subsets of~$M$ and  such that $M\notin J$.  Let us emphasize that ideals on $\omega$ or on any countable set can not be $\sigma$-ideals. We focus on ideals on $\omega$ and we use the letters $\Jcal$ and $\Kcal$ exclusively to denote such ideals. For $A\subseteq\pts(M)$ we denote 
\[
A^{d}=\set{a\subseteq M}{M\smallsetminus a\in\Acal}.
\]
Recall that $\Fcal\subseteq\pts(M)$ is a~\emph{filter} when $\Fcal^d$ is an~ideal. A~maximal filter $\Ucal\subseteq\pts(M)$ with respect to inclusion is called an~\emph{ultrafilter}. For an~ideal $\Kcal\subseteq\pts(M)$ we denote $\Kcal^{+}=\pts(M)\smallsetminus\Kcal$. One can see that $a\in\Kcal^+$ if and only if $M\smallsetminus a\notin \Kcal^d$.

A set $A\subseteq\pts(M)$ \emph{generates an ideal on $M$} iff it has the so called \emph{finite union property}, i.e. $M\smallsetminus \bigcup C$ is infinite for any finite $C\subseteq A$. In this case, the ideal generated by $A$ is\footnote{Considering that an ideal must contain all finite sets.}
\[\set{a\subseteq M}{a\smallsetminus \bigcup C \text{ is finite for some finite }C\subseteq A}.\]

When $s$ and $t$ are functions (or sequences $s=\la s_i:\, i\in a\ra$ and $t:=\la t_i:\, i\in b\ra$), $s\subseteq t$ means that $s$ extends $t$, i.e. $\dom s\subseteq \dom t$ and $t\frestr \dom s=s$ (or, $a\subseteq b$ and $s_i=t_i$ for all $i\in a$).
We denote by $\mu$ the~Lebesgue measure defined on~the Cantor space ${}^\omega2$, that is, the (completion of the) product measure on ${}^\omega 2=\prod_{n<\omega}2$ where $2=\{0,1\}$ is endowed with the probability measure that sets $\{0\}$ of measure $\frac{1}{2}$. In fact, for any $s\in 2^{<\omega}$, $\mu([s])=2^{-|s|}$ where $|s|$ denotes the length of $s$ and $[s]:=\set{x\in \cantor}{s\subseteq x}$. Recall that $\set{[s]}{s\in 2^{<\omega}}$ is a base of clopen sets of the topology of $\cantor$.
Then
\[
    \mathcal{N}:=\set{A\subseteq {}^\omega2}{\leb{A}=0} \text{ \quad and \quad } \mathcal{M}:=\set{A\subseteq {}^\omega2}{A \text{ is meager}}
\]
are $\sigma$-ideals, i.e.\ the union of any countable subset of the ideal belongs to the ideal.

The ideal $\Ncal$ has a combinatorial characterization in terms of clopen sets. To present this, we fix the following terminology.
We usually write $\bar{c} :=\seqn{c_n}{n\in\omega}$ for sequences of sets.

\begin{definition}\label{def:clpN}
Denote $\Omega:=\set{c\subseteq {}^\omega2}{c \text{ is a~clopen set}}$. For $\varepsilon\colon \omega\to(0,\infty)$, consider the set
\[
\Omega^*_\varepsilon := \set{\bar{c}\in{}^\omega\Omega}{(\forall\, n\in\omega)\ \leb{c_n}\leq \varepsilon_n}.
\]
For each $\bar{c}\in{}^\omega\Omega$ (or in ${}^\omega\pts(\omega)$ in general) denote 
\[
\begin{split}
N(\bar{c}) & :=\bigcap_{m<\omega}\bigcup_{n\geq m}c_n =\set{x\in {}^\omega2}{|\set{n\in\omega}{x\in c_n}|=\aleph_0},\\
N^*(\bar c) & :=\bigcup_{m<\omega}\bigcap_{n\geq m}c_n =\set{x\in\cantor}{|\set{n\in\omega}{x\notin c_n}|<\aleph_0}.
\end{split}
\]
\end{definition}

\begin{fact}[{\cite[Lemma~2.3.10]{BJ}}]\label{basicN2}
Let $\varepsilon:\omega\to(0,\infty)$, and assume that $\sum_{i\in\omega}\varepsilon_i<\infty$. Then:
\begin{enumerate}[label=\rm (\alph*)]
    \item For any $\bar{c}\in\Omega^*_\varepsilon$, $N(\bar{c})\in\mathcal{N}$.
    \item For any $X\subseteq {}^\omega2$, $X\in\mathcal{N}$ iff $(\exists\, \bar{c}\in\Omega^*_\varepsilon)\ X\subseteq N(\bar{c})$.
\end{enumerate}
\end{fact}

The analog of \autoref{basicN2} using $N^*(\bar c)$ becomes the characterization of $\Ecal$, the $\sigma$-ideal generated by the closed measure zero subsets of $\cantor$.

\begin{fact}\label{basicE}
Let $\varepsilon:\omega\to(0,\infty)$ and assume that $\liminf_{i\to\infty}\varepsilon_i=0$. Then:
\begin{enumerate}[label=\rm (\alph*)]
    \item\label{it:inE} For any $\bar{c}\in\Omega^*_\varepsilon$, $N^*(\bar{c})\in\mathcal{E}$.
    \item\label{it:charE} For any $X\subseteq {}^\omega2$, $X\in\mathcal{E}$ iff $(\exists\, \bar{c}\in\Omega^*_\varepsilon)\ X\subseteq N^*(\bar{c})$.
\end{enumerate}
\end{fact}
\begin{proof}
Item~\ref{it:inE} is clear because, for any $n\in\omega$, $\bigcap_{m\geq n}c_m$ is closed and, since $\liminf\limits_{i\to\infty}\varepsilon_i=0$ and $\bar c\in\Omega^*_\varepsilon$, it has measure zero.

For~\ref{it:charE}, the implication $\Leftarrow$ is clear by~\ref{it:inE}. To see $\Rightarrow$, let $X\in\Ecal$, i.e. $X\subseteq \bigcup_{n<\omega}F_n$ for some increasing sequence $\la F_n:\, n<\omega\ra$ of closed measure zero sets. For each $n<\omega$, we can cover $F_n$ with countably many basic clopen sets $[s_{n,k}]$ ($k<\omega$) such that $\sum\limits_{k<\omega}\mu([s_{n,k}])<\varepsilon_n$, but by compactness only finitely many of them cover $F_n$, so $F_n\subseteq c_n:=\bigcup_{k<m_n}[s_{n,k}]$ for some $m_n<\omega$, and $\mu(c_n)<\varepsilon_n$. Then $\bar c:=\la c_n:\, n<\omega\ra$ is as required.
\end{proof}

Motivated by the~combinatorial characterization of $\Ncal$ and $\Ecal$ presented in \autoref{basicN2} and~\ref{basicE}, we introduce a smooth modification of these via ideals on~$\omega$. To start, we fix more terminology and strengthen the previous characterizations.

\begin{definition}\label{def:omegabar}
Denote $\omegabar :=\set{\bar{c}\in{}^\omega\Omega}{ N(\bar{c})\in \Nideal}$.
\end{definition}

By \autoref{basicN2}~\ref{it:inE} we have that $\Omega^*_\varepsilon\subseteq\omegabar$ whenever $\varepsilon\colon\omega\to(0,\infty)$ and $\sum_{i<\omega}\varepsilon_i<\infty$. Hence, as a direct consequence of \autoref{basicN2}:

\begin{fact}\label{basicNE}
For any $X\subseteq\cantor$, $X\in\Ncal$ iff $X\subseteq N(\bar c)$ for some $\bar c\in\omegabar$.
\end{fact}

We also have the analogous version of $\Ecal$. Before stating it, we characterize $\omegabar$ as follows.

\begin{lemma}\label{chomegabar}
 For any sequence $\bar c=\la c_i:\, i<\omega\ra$, the following statements are equivalent.
 \begin{enumerate}[label=\rm (\roman*)]
     \item\label{it:inombar} $\bar c\in\omegabar$.
     \item\label{it:borelch} $(\forall\, \varepsilon\in \Q^+)\ (\exists\, N<\omega)\ (\forall\, n\geq N)\  \mu\left(\bigcup_{N\leq i<n}c_i\right)<\varepsilon$
     \item\label{it:anypart} For any $\varepsilon\colon \omega\to(0,\infty)$ there is some interval partition $\bar I=\la I_n:\, n<\omega\ra$ of $\omega$ such that $\mu\left(\bigcup_{i\in I_n}c_i\right)<\varepsilon_n$ for all $n>0$.
     \item\label{it:expart} There is an interval partition $\bar I=\la I_n:\, n<\omega\ra$ of $\omega$ such that \[\sum_{n<\omega}\mu\left(\bigcup_{i\in I_n}c_i\right)<\infty.\]
 \end{enumerate}
\end{lemma}
\begin{proof}
$\ref{it:inombar} \imp \ref{it:borelch}$: Let $\varepsilon$ be a positive rational number. Since $\bar c\in\omegabar$, $\mu(N(\bar c))=0$, so $\lim\limits_{n\to\infty}\mu\left(\bigcup_{i\geq n}c_i\right)=0$. Then $\mu\left(\bigcup_{i\geq N}c_i\right)<\varepsilon$ for some $N<\omega$, which clearly implies that $\mu\left(\bigcup_{N\leq i<n}c_i\right)<\varepsilon$ for all $n\geq N$.

$\ref{it:borelch}\imp\ref{it:anypart}$: Let $\varepsilon\colon \omega\to(0,\infty)$. Using~\ref{it:borelch}, by recursion on $n<\omega$ we define an increasing sequence $\la m_n:\, n<\omega\ra$ with $m_0=0$ such that $\mu\left(\bigcup_{m_{n+1}\leq i<k}c_i\right)<\varepsilon_{n+1}$ for all $k\geq m_{n+1}$. Then $I_n:=[m_n,m_{n+1})$ is as required.

$\ref{it:anypart}\imp\ref{it:expart}$: Apply~\ref{it:anypart} to $\varepsilon_n:=2^{-n}$.

$\ref{it:expart}\imp\ref{it:inombar}$: Choose $\bar I$ as in~\ref{it:expart}, and let $c'_n:=\bigcup_{i\in I_n}c_i$. It is clear that $N(\bar c)=N(\bar c')$ and $\bar c'\in\Omega^*_\varepsilon$ where $\varepsilon\colon\omega\to(0,\infty)$, $\varepsilon_n:=\mu(c'_n)+2^{-n}$. Since $\sum_{n<\omega}\mu(c'_n)<\infty$, by \autoref{basicN2} we obtain that $N(\bar c)=N(\bar c')\in\Ncal$. Thus $\bar c\in\omegabar$.
\end{proof}

As a consequence of \autoref{chomegabar}~\ref{it:borelch}, considering $\Omega$ as a countable discrete space:

\begin{corollary}\label{cor:omegabarBorel}
  The set $\omegabar$ is Borel in ${}^\omega\Omega$.
\end{corollary}

\begin{lemma}\label{charE}
  The ideal $\Ecal$ is characterized as follows.
  \begin{enumerate}[label=\rm (\alph*)]
      \item\label{it:N*c} $N^*(\bar c)\in\Ecal$ for any $\bar c\in\omegabar$.
      \item\label{it:Ech} For $X\subseteq\cantor$, $X\in\Ecal$ iff $(\exists\, \bar c\in\omegabar)\ X\subseteq N^*(\bar c)$.
  \end{enumerate}
\end{lemma}
\begin{proof}
\ref{it:N*c} is proved similarly as \autoref{basicE}~\ref{it:inE}, noting that $\bar c\in\omegabar$ implies that $\lim_{i\to\infty}\mu(c_i)=0$ (by \autoref{chomegabar}~\ref{it:borelch}). \ref{it:Ech} follows by~\ref{it:N*c} and \autoref{basicE}.
\end{proof}

We use this characterization to introduce the promised generalized versions of $\Ncal$ and $\Ecal$.
Consider the ideal $\Fin$ of finite subsets of $\omega$. For $\bar c\in{}^\omega\Omega$ and $x\in\cantor$, note that,
\begin{linenomath}
\begin{align*}
    x\in N(\bar c) & \sii \set{n\in\omega}{x\in c_n}\in\Fin^+,\\
    x\in N^*(\bar c) & \sii \set{n\in\omega}{x\in c_n}\in\Fin^d.
\end{align*}
\end{linenomath}
Replacing $\Fin$ by an arbitrary ideal on $\omega$, we obtain the following notion.

\begin{definition}\label{def:NJ}
Fix an ideal $\Jcal$ on $\omega$. For $\bar c\in{}^\omega\Omega$, define
\begin{linenomath}
\begin{align*}
 N_{\Jcal}(\bar{c})&:=\set{x\in {}^\omega2}{\set{n\in\omega}{x\in c_n}\in\Jcal^+},\\
 N^*_{\Jcal}(\bar{c})&:=\set{x\in {}^\omega2}{\set{n\in\omega}{x\in c_n}\in \Jcal^d}.
\end{align*}
\end{linenomath}
These sets are used to define the families
\begin{linenomath}
\begin{align}\label{Nj2}
    \NidealJ{\Jcal}& := \set{X\subseteq {}^\omega2}{(\exists\, \bar{c}\in \omegabar)\ X\subseteq N_{J}(\bar{c})},\\ \label{Nstarj2}
     \Ncal^*_{\Jcal}&:=\set{X\subseteq {}^\omega2}{(\exists\, \bar{c}\in\omegabar)\ X\subseteq N^*_{\Jcal}(\bar{c})}.
\end{align}
\end{linenomath}
We say that the members of $\Ncal_\Jcal$ \emph{have measure zero (or are null) modulo $\Jcal$}.
\end{definition}

Due to \autoref{basicNE} and~\autoref{charE}, we obtain $\NidealJ{\fin}=\Ncal$ and $\Ecal=\Ncal^*_{\Fin}$. 
Moreover, one can easily see that, if $\Jcal\subseteq\Kcal$ are ideals on $\omega$, then 
\begin{equation}\label{eq:cont}
\Ncal^*_{\Jcal}\ \subseteq\ \Ncal^*_{\Kcal}\ \subseteq\ \Ncal_\Kcal\ \subseteq\ \Ncal_\Jcal.    
\end{equation}
In particular, we obtain 
\begin{equation}\label{eq:contfin}
\Ecal=\Ncal^*_\fin\subseteq\Ncal^*_{\Jcal}\subseteq\Ncal_\Jcal\subseteq\Ncal_\fin=\Ncal.    
\end{equation}
Furthermore, if $\Jcal$ is a~maximal ideal (i.e.\ its dual $\Jcal^d$ is an~ultrafilter), then $\NidealJ{\Jcal}=\NstaridealJ{\Jcal}$.

We can prove that, indeed, both $\Ncal_\Jcal$ and $\Ncal^*_\Jcal$ are $\sigma$-ideals on the reals, exactly as the original notions.

\begin{lemma}\label{NJsigma2}
Let $\Jcal$ be an ideal on~$\omega$. Then both
$\NidealJ{\Jcal}$ and $\NstaridealJ{\Jcal}$ are $\sigma$-ideals on~$\cantor$.
\end{lemma}
\begin{proof}
Thanks to \eqref{eq:contfin}, both $\NidealJ{\Jcal}$ and $\NstaridealJ{\Jcal}$ contain all finite subsets of $\cantor$ and the whole space $\cantor$ does not belong to them. It is also clear that both families are downwards closed under $\subseteq$, so
it is enough to verify that both $\NidealJ{\Jcal}$ and $\NstaridealJ{\Jcal}$ are closed under countable unions.

Consider $\bar{c}^k \in\omegabar$ for each $k\in\omega$. By recursion, using \autoref{chomegabar}~\ref{it:borelch}, we define an increasing sequence $\seqn{n_l}{l<\omega}$ of natural numbers such that 
$\mu\left(\bigcup_{n\geq n_l}c^k_n\right)<\frac{1}{(l+1)2^l}$ for all $k\leq l$. Let $I_0:=[0,n_1)$ and $I_l:=[n_l,n_{l+1})$ for $l>0$.
Then, we define the~sequence $\bar{c}$ by
\begin{linenomath}
\begin{equation*}
 c_n = \begin{cases}
  \bigcup\limits_{k\leq l}c^k_n & \text{if } n\in[n_l,n_{l+1}),\\[2ex]
 \emptyset & \text{if }n<n_0.
 \end{cases}
\end{equation*}
\end{linenomath}
Finally,
\[\sum_{l<\omega}\mu\left(\bigcup_{n\in I_l}c_n\right)\leq \sum_{l<\omega}\sum_{k\leq l}\mu\left(\bigcup_{n\geq n_l}c^k_n\right)\leq \sum_{l<\omega}\sum_{k\leq l}\frac{1}{(l+1)2^l}=\sum_{l<\omega}\frac{1}{2^l}<\infty,\]
so $\bar c\in\omegabar$ by \autoref{chomegabar}~\ref{it:expart}. It is clear that
$\bigcup_{k\in\omega}N_\Jcal(\bar{c}^k)\subseteq N_\Jcal(\bar{c})$ and $\bigcup_{k\in\omega}N^*_\Jcal(\bar{c}^k)\subseteq N^*_\Jcal(\bar{c})$.
\end{proof}

\begin{remark}\label{rem:opennonew}
  The following alternative definition does not bring anything new: Let $\omegabar_0$ be the set of countable sequences $\bar a=\la a_n:n<\omega\ra$ of open subsets of ${}^\omega2$ such that $N(\bar a)\in\Ncal$. Define $N_\Jcal(\bar{a})$ similarly, and $\Ncal^0_\Jcal$ as the family of subsets of ${}^\omega2$ that are contained in some set of the form $N_\Jcal(\bar a)$ for some $\bar a\in\omegabar_0$. Define $\Ncal^{*0}_\Jcal$ analogously.
   It is not hard to show that $\Ncal^{*0}_\Jcal=\Ncal^0_\Jcal=\Ncal$. The inclusions $\Ncal^{*0}_\Jcal\subseteq\Ncal^0_\Jcal\subseteq\Ncal$ are clear; 
   to see $\Ncal\subseteq\Ncal^{*0}_\Jcal$, if $B\in\Ncal$, then we can find some $\bar a\in\omegabar_0$ such that $B\subseteq\bigcap_{n\in\omega}a_n$ and $\mu(a_n)<2^{-n}$, so it is clear that $B\subseteq N^*_\Jcal(\bar a)$. For this reason, it is uninteresting to consider sequences of open sets instead of clopen sets.   
\end{remark}

\begin{remark}\label{rem:expand}
    We expand our discussion by allowing ideals on an arbitrary infinite countable set $W$ instead of $\omega$. Namely, for $\bar c \in {}^W \Omega$ (or in ${}^W \pts(\cantor)$), we define $N_J(\bar c)$ and $N^*_J(\bar c)$ similar to \autoref{def:NJ}, $\Fin(W)$ as the ideal of finite subsets of $W$,
    \[\omegabar(W):=\set{\bar c \in {}^W \Omega}{N_{\Fin(W)}(\bar c)\in\Ncal} \text{ (so $\omegabar(\omega)= \omegabar$)}\]
    and $\Ncal_J$ and $\Ncal^*_J$ as in \eqref{Nj2} and~\eqref{Nstarj2}, respectively. We need this expansion to allow ideals obtained by operations as in~\autoref{exm:sum}.

    Using a one-to-one enumeration $W=\set{w_n}{n<\omega}$ by \autoref{chomegabar} we get that, for $\bar c\in {}^W \Omega$,
    \begin{multline}\label{eq:OmW}
      \bar c\in \omegabar(W) \text{ iff, }\\
      \text{for any }\varepsilon>0 \text{, there is some finite set }a\subseteq W\text{ such that }\mu\left(\bigcup_{n\in W\smallsetminus a} c_n\right)<\varepsilon.  
    \end{multline}
    In fact, considering the bijection $f\colon \omega\to W$ defined by $f(n):=w_n$ for $n<\omega$, and the ideal $J':=\set{f^{-1}[a]}{a\in J}$ (which is isomorphic to $J$), we obtain $\Ncal_J = \Ncal_{J'}$ and $\Ncal^*_J =\Ncal^*_{J'}$. 
    As a consequence, $\Ncal_{\Fin(W)}=\Ncal_\Fin=\Ncal$ and $\Ncal^*_{\Fin(W)}=\Ncal^*_{\Fin}=\Ecal$.
\end{remark}

\begin{example}\label{exm:sum}
Consider $\omega=\N_1\cup\N_2$ as a disjoint union of infinite sets, let $J_1$ be an ideal on $\N_1$ and $J_2$ an ideal on $\N_2$. Recall the ideal \[J_1\oplus J_2=\set{x\subseteq\omega}{x\cap\N_1\in J_1 \text{ and } x\cap\N_2 \in J_2}.\]
Note that
\begin{enumerate}[label =\rm (\arabic*)]
    \item\label{it:omN1} $\bar c\in\omegabar$ iff $\bar c\frestr\N_1\in\omegabar(\N_1)$ and $\bar c\frestr\N_2\in\omegabar(\N_2)$.
    \item\label{it:Ncplus} For $\bar c\in\omegabar$,
    \begin{linenomath}
    \begin{align*}
    N_{J_1\oplus J_2}(\bar c) & =N_{J_1}(\bar c\frestr\N_1)\cup N_{J_2}(\bar c\frestr\N_2)) \text{ and }\\
    N^*_{J_1\oplus J_2}(\bar c) & =N^*_{J_1}(\bar c\frestr\N_1)\cap N^*_{J_2}(\bar c\frestr\N_2).
    \end{align*}
    \end{linenomath}
\end{enumerate}

As a consequence:
\begin{enumerate}[resume*]
    \item\label{it:Nplus} $\Ncal_{J_1\oplus J_2}$ is the ideal generated by $\Ncal_{J_1}\cup \Ncal_{J_2}$, in fact
    \[\Ncal_{J_1\oplus J_2}=\set{X\cup Y}{X\in \Ncal_{J_1} \text{ and } Y\in \Ncal_{J_2} }.\]

    \item\label{it:N*plus} $\Ncal^*_{J_1\oplus J_2}=\Ncal^*_{J_1}\cap \Ncal^*_{J_2}$.
\end{enumerate}
The inclusion $\subseteq$ in both~\ref{it:Nplus} and~\ref{it:N*plus} follows from~\ref{it:Ncplus}; the converse follows by the fact that, whenever $\bar c^1\in\omegabar(\N_1)$ and $\bar c^2\in\omegabar(\N_2)$, $\bar c\in\omegabar$ where $\bar c = \bar c^1\cup \bar c^2$, i.e. $\bar c=\la c_n:\, n<\omega\ra$ such that $c_n:=c^i_n$ when $n\in\N_i$ ($i\in\{1,2\}$), and
\begin{linenomath}
    \begin{align*}
    N_{J_1\oplus J_2}(\bar c) & =N_{J_1}(\bar c^1)\cup N_{J_2}(\bar c^2)) \text{ and }\\
    N^*_{J_1\oplus J_2}(\bar c) & =N^*_{J_1}(\bar c^1)\cap N^*_{J_2}(\bar c^2),
    \end{align*}
    \end{linenomath}
which follow by~\ref{it:omN1} and~\ref{it:Ncplus} because $\bar c^i=\bar c\frestr\N_i$ for $i\in\{1,2\}$.

By allowing $\pts(\omega)$ instead of an ideal, $N_{\pts(\omega)}(\bar c)=\emptyset$ and $N^*_{\pts(\omega)}(\bar c)=\cantor$. Since
\[J_1\oplus\pts(\N_2)=\set{a\subseteq\omega}{a\cap\N_1\in J_1}\]
is an ideal, for any $\bar c\in\omegabar$, we obtain $N_{J_1\oplus \pts(\N_2)}(\bar c)  =N_{J_1}(\bar c\frestr\N_1)$ and $N_{J_1\oplus \pts(\N_2)}^*(\bar c)  =N_{J_1}^*(\bar c\frestr\N_1)$ by~\ref{it:Ncplus}. Therefore, $\Ncal_{J_1\oplus \pts(\N_2)}=\Ncal_{J_1}$ and $\Ncal^*_{J_1\oplus \pts(\N_2)}=\Ncal^*_{J_1}$. Similar conclusions are valid for $\pts(\N_1)\oplus J_2$.



\end{example}

We now review the following classical orders on ideals.

\begin{definition}\label{def:RB}
Let $M_1$ and $M_2$ be infinite sets, $\Kcal_1\subseteq\pts(M_1)$ and $\Kcal_2\subseteq\pts(M_2)$. If $\varphi\colon M_2\to M_1$, the 
\emph{projection of $\Kcal_2$ under $\varphi$} is the family 
\[
\varphi^{\to}(\Kcal_2)=\set{A\subseteq M_1}{\varphi^{-1}(A)\in\Kcal_2}.
\]
We write $\Kcal_1\leq_{\varphi}\Kcal_2$ when $\Kcal_1\subseteq\varphi^{\to}(\Kcal_2)$, i.e., $\varphi^{-1}(I)\in\Kcal_2$ for any $I\in\Kcal_1$. 

\begin{enumerate}[label = \rm (\arabic*)]
    \item \cite{katetov,blass86}  The \emph{Kat\v{e}tov-Blass order} is defined by $\Kcal_1\leq_{\text{KB}}\Kcal_2$ iff there is a~finite-to-one function $\varphi\colon M_2\to M_1$ such that $I\in\Kcal_1$ implies $\varphi^{-1}(I)\in\Kcal_2$, i.e.\ $\Kcal_1\subseteq\varphi^{\to}(\Kcal_2)$.
    
    \item \cite{blass86,LZ} The \emph{Rudin-Blass order} is defined by $\Kcal_1\leq_{\text{RB}}\Kcal_2$ iff there is a~finite-to-one function $\varphi\colon M_2\to M_1$ such that $I\in\Kcal_1$ if and only if $\varphi^{-1}(I)\in\Kcal_2$, i.e.\ $\varphi^{\to}(\Kcal_2)=\Kcal_1$.
    
    \item We also consider the ``dual" of the Kat\v{e}tov-Blass order: $K_1\leqKBpr K_2$ iff there is some finite-to-one function $\varphi\colon M_1\to M_2$ such that $\varphi^\to(K_1)\subseteq K_2$.
\end{enumerate}

Recall that the relations $\leqKB$ and $\leqRB$ are reflexive and transitive, and it can be proved easily that $\leqKBpr$ also has these properties.

Note that $K_1\leqRB K_2$ implies $K_1\leqKB K_2$ and $K_2\leqKBpr K_1$. Also recall that $K_1\subseteq K_2$ implies $K_1\leqKB K_2$ and $K_1\leqKBpr K_2$ (using the identity function).

Recall that, if $\Kcal_2$ is an~ideal on $M_2$, then $\varphi^{\to}(\Kcal_2)$ is downwards closed under $\subseteq$ and closed under finite unions, and $M_1\not\in\varphi^{\to}(\Kcal_2)$. If $\varphi$ is in addition finite-to-one then $\varphi^{\to}(\Kcal_2)$ is an ideal.
\end{definition}

We show that our defined $\sigma$-ideals behave well under the previous orders.
In fact, this is a somewhat expected result that can usually be obtained for many well-known objects in topology. For example,
given an ideal $J$ on $\omega$, consider the relation $\leq^J$ on $\baire$ defined by $x\leq^J y$ iff $\set{n<\omega}{x(n) \nleq y(n)}\in J$, and define the cardinal characteristics
\begin{align*}
    \bfrak_J & := \min\set{|F|}{F\subseteq\baire,\ \neg(\exists\, y\in\baire)\ (\forall\, x\in F)\ x\leq^J y},\\
    \dfrak_J & := \min\set{|D|}{D\subseteq\baire,\ (\forall\, x\in\baire)\ (\exists\, y\in D)\ x\leq^J y}.
\end{align*}
It is known from \cite{FaSo} that:
\begin{enumerate}[label = \rm (\arabic*)]
    \item If $K\leqKB J$ then $\bfrak_K\leq \bfrak_J$ and $\dfrak_J\leq \dfrak_K$.
\item If $K \leqKBpr J$ then $\bfrak_K\leq \bfrak_J$ and $\dfrak_J\leq \dfrak_K$.
\item If $K\leqRB J$ then $\bfrak_I=\bfrak_J$ and $\dfrak_I=\dfrak_J$.
\end{enumerate}
We present another similar example in \autoref{KBomegabar}.\footnote{More similar examples of such implications can be found in~\cite{Hrusak,SS,Miro2}. 
}

\begin{theorem}\label{RBeq}
  Let $\N_1$ and $\N_2$ be countable infinite sets, 
  $\Jcal$ an ideal on $\N_1$ and let $\Kcal$ be an ideal on $\N_2$. Then
  \begin{enumerate}[label=\rm(\alph*)]
      \item\label{it:underKB} If $\Kcal\leqKB \Jcal$ then $\Ncal^*_\Kcal\subseteq\Ncal^*_\Jcal$ and $\Ncal_\Jcal\subseteq\Ncal_\Kcal$.
      \item\label{it:underKB'} If $K\leqKBpr J$ then $\Ncal^*_K\subseteq \Ncal^*_J$ and $\Ncal_J\subseteq \Ncal_K$.
      \item\label{it:underRB} If $\Kcal\leqRB \Jcal$ then $\Ncal^*_\Jcal=\Ncal^*_\Kcal$ and $\Ncal_\Jcal=\Ncal_\Kcal$.
  \end{enumerate}
\end{theorem}
\begin{proof}
Without loss of generality, we may assume $\N_1=\N_2=\omega$ in this proof.
Fix a finite-to-one function $f\colon \omega\to\omega$ and let $I_n:=f^{-1}[\{n\}]$ for any $n\in\omega$. Given $\bar c\in \omegabar$ we define sequences $\bar c'$ and $\bar c^-$ by $c'_n:=\bigcup_{k\in I_n}c_k$ and $c^-_k:= c_{f(k)}$. Since $N(\bar c')= N(\bar c)$ and $N(\bar c^-)\subseteq N(\bar c)$, we have that $\bar c',\bar c^- \in \omegabar$.

It is enough to show 
\begin{enumerate}[label= \rm (\roman*)]
    \item\label{it:conKB} $K\subseteq f^\to(J)$ implies $N^*_K(\bar c) \subseteq N^*_J(\bar c^-)$ and $N_J(\bar c) \subseteq N_K(\bar c')$, and
    \item\label{itconKB'} $f^\to(K)\subseteq J$ implies $N^*_K(\bar c)\subseteq N^*_J(\bar c')$ and $N_J(\bar c)\subseteq N_K(\bar c^-)$.
\end{enumerate}

\ref{it:conKB}: Assume  $K\subseteq f^\to(J)$. If $x\in N^*_\Kcal(\bar c)$ then $\{n<\omega:\, x\notin c_n\}\in\Kcal$, so
\[\{k<\omega:\, x\notin c^-_k\}=f^{-1}[\set{n<\omega}{x\notin c_n}]\in\Jcal,\]
i.e.\ $x\in N^*_\Jcal(\bar c^-)$; and if $x\in N_\Jcal(\bar c)$, i.e.\ $\{k<\omega:\, x\in c_k\}\notin\Jcal$, then
\[
    f^{-1}[\{n<\omega:\, x\in c'_n\}] = \{k<\omega:\, x\in c'_{f(k)}\}\supseteq\{k<\omega:\, x\in c_k\},
\]
so  $f^{-1}[\{n<\omega:\, x\in c'_n\}]\notin\Jcal$, i.e.\ $\{n<\omega:\, x\in c'_n\}\notin\Kcal$, which means that $x\in N_\Kcal(\bar c')$.

\ref{itconKB'} Assume $f^\to(K)\subseteq J$. If $x\in N^*_K(\bar c)$, i.e.\ $\{k<\omega:\, x\notin c_k\}\in K$, then
\[
    f^{-1}[\{n<\omega:\, x\notin c'_n\}] = \{k<\omega:\, x\notin c'_{f(k)}\}\subseteq\{k<\omega:\, x\notin c_k\},
\]
so $f^{-1}[\{n<\omega:\, x\notin c'_n\}]\in K$, which implies that $\{n<\omega:\, x\notin c'_n\}\in J$, i.e.\ $x\in N^*_J(\bar c')$; and if $x\in N_J(\bar c)$ then $\{n<\omega:\, x\in c_n\}\notin J$, so
\[\{k<\omega:\, x\in c^-_k\}=f^{-1}[\set{n<\omega}{x\in c_n}]\notin K,\]
i.e.\ $x\in N_K(\bar c^-)$.
\end{proof}

\begin{example}\label{exm:sum2}
Notice that $J\leqRB J\oplus\pts(\omega)$, however, $J$ and $\pts(\omega)$ should come from different sets. Concretely, if $J$ is an ideal on $\omega$ then $J\oplus\pts(\omega)$ should be formally taken as $J\oplus\pts(\N')$ where $\N'$ is an infinite countable set and $\omega\cap \N'=\emptyset$. Therefore, by \autoref{RBeq}, $\Ncal_{J\oplus\pts(\omega)}=\Ncal_J$ and $\Ncal^*_{J\oplus\pts(\omega)}=\Ncal^*_J$ (already known at the end of \autoref{exm:sum}).

\end{example}

Recall the following well-known result that characterizes ideals on $\omega$ with the Baire property.

\begin{theorem}[Jalani-Naini and Talagrand~{\cite{Talagrand}}]\label{thm:talagrand}
Let $\Jcal$ be an ideal on $\omega$. Then the following statements are equivalent.

\begin{multicols}{2}
\begin{enumerate}[label=\rm (\roman*)]
   \item $\Jcal$ has the Baire property in $\pts(\omega)$.
   \item $\Jcal$ is meager in $\pts(\omega)$.
   \item $\Fin\leqRB \Jcal$.
\end{enumerate}
\end{multicols}
\end{theorem}

Therefore, as a consequence of \autoref{RBeq}:

\begin{theorem}\label{BaireForNJ2}
If $\Jcal$ is an ideal on $\omega$ with the Baire property, then $\Ncal_\Jcal=\Ncal$ and $\Ncal^*_\Jcal=\Ecal$.  
\end{theorem}

So the $\sigma$-ideals associated with definable (analytic) ideals do not give new $\sigma$-ideals on $\cantor$, but they give new characterizations of $\Ncal$ and $\Ecal$. The converse of the previous result is also true, which we fully discuss in the next section.

\section{Ideals without the Baire property}\label{sec:nBP}

In this section, we study the ideals $\Ncal_J$ and $\Ncal^*_J$ when $J$ does not have the Baire property. With respect to our main results,
we finish to prove \autoref{mainBP}: an ideal $\Jcal$ on $\omega$ without the Baire property gives us new $\sigma$-ideals $\Ncal_\Jcal$ and $\Ncal^*_\Jcal$. We prove \autoref{mainBP2} as well (see \autoref{conseqforE2}).

One of the~main tools in our study of $\NidealJ{\Jcal}$ and $\NstaridealJ{\Jcal}$ is the~technique of \emph{filter games}.\footnote{One can find a~good systematic treatment of filter games and their dual ideal versions in~\cite{Laf, Laflear}.}
For the main results mentioned in the previous paragraph, we will use the meager game.

\begin{definition}\label{def:nonMgame}
Let $\Fcal$ be a~filter on~$\omega$. The following game of length $\omega$ between two players is called the \emph{meager game $M_\Fcal$}:
\begin{itemize}
    \item In the $n$-th move, Player~I plays a~finite set $A_n\in[\omega]^{<\omega}$ 
    and Player~II responds with a~finite set $B_n\in[\omega]^{<\omega}$ disjoint from $A_n$.
    \item After $\omega$ many moves, Player~II wins if $\bigcup\set{B_n}{n\in\omega}\in\Fcal$, and Player~I  wins otherwise.
\end{itemize}
\end{definition}   

Let us recall an~important result from Bartoszy\'nski and Scheepers about the~aforementioned game.\footnote{See also~{\cite[Thm.~2.11]{Laf}}.}

\begin{theorem}[{\cite{BSch}}]\label{nonMgame}
Let $\Fcal$ be a~filter on~$\omega$. Then
Player~I does not have a~winning strategy in the meager game for the filter $\Fcal$ if and only if $\Fcal$ is not meager in $\pts(\omega)$.
\end{theorem}

We use the meager game to show that, whenever $J$ does not have the Baire property, no member of $\Ncal_J$ can be co-meager with respect to any \emph{self-supported} closed subset of $\cantor$, i.e.\ a closed subset of positive measure such that each of its non-empty (relative) open subsets have positive measure.

\begin{mainlemma}\label{notcomeager2}
Let $C\subseteq{}^\omega 2$ be a self-supported closed set. 
If $\Jcal$ is not meager then $\displaystyle N_\Jcal(\bar{c})\cap C$ is not co-meager in $C$ for each $\bar{c}\in\omegabar$. As a consequence, $Z\cap C$ is not co-meager in $C$ for any $Z\in\Ncal_\Jcal$.
\end{mainlemma}
\begin{proof}
Let $C\subseteq{}^\omega 2$ be a~self-supported closed set, and 
let $G\subseteq C$ be a~co-meager subset in $C$. Then there is a~sequence $\seqn{D_n}{n\in\omega}$ of open dense sets in $C$ such that $\bigcap_{n\in\omega}D_n\subseteq G$. Moreover, since $C$ is closed, there is a~tree $T$ such that $C=[T]$. 

Consider $\bar{c}\in\omegabar$ and construct the~following strategy of Player~I for the meager game for $\Jcal^d$.\smallskip

\underline{The~first move}:

Player~I picks an $s_0\in T$ such that $[s_0]\cap C\subseteq D_0$ and chooses $n_0<\omega$ such that $\leb{\bigcup_{n\geq n_0}c_n}<\leb{C\cap[s_0]}$ (which exists because $\bar c\in\omegabar$). 
Player~I's move is $n_0$.\smallskip
 
\underline{Second move and further}:

Player~II replies with $B_0\in[\omega]^{<\omega}$ such that $n_0\cap B_0=\emptyset$.
Since $\leb{\bigcup_{n\in B_0} c_n}<\leb{C\cap[s_0]}$, $C\cap [s_0]\nsubseteq \bigcup_{n\in B_0}c_n$, so there exists an $x_0\in C\cap [s_0]\smallsetminus \bigcup_{n\in B_0} c_n$. Then Player~I finds an
$m_0>|s_0|$ such that $[x_0{\upharpoonright}m_0]\cap \bigcup_{n\in B_0} c_n=\emptyset$. 
Since $D_1$ is dense in $C$, Player~I can pick an $s_1\in T$ such that $s_0\subseteq x_0{\upharpoonright} m_0\subseteq s_1$ and $[s_1]\cap C\subseteq D_1$. Then Player~I moves with an $n_1\in\omega$ such that $\leb{\bigcup_{n\geq n_1} c_n}<\leb{C\cap [s_1]}$.  
Player~II responds with $B_1\in[\omega]^{<\omega}$ such that $n_1\cap B_1=\emptyset$, and the~game continues in the~same way we just described. 

Finally, since $\Jcal$ is not meager, Player~I does not have a~winning strategy in the meager game for $\Jcal^d$. Hence, there is some match $\seqn{(n_k, B_k)}{k\in\omega}$ where Player~I uses the~aforementioned strategy and Player~II wins. Thus, we have $F:=\bigcup_{k\in\omega} B_k\in\Jcal^d$. 
Define $x:=\bigcup_{k\in\omega}s_k$. Then, $x\in\bigcap_{k\in\omega}D_k$. On the~other hand, $x\notin \bigcup_{n\in F}c_n$ (because $[s_{k+1}]\cap\bigcup_{n\in B_k}c_k=\emptyset$) and hence $x\notin N_{\Jcal}(\bar{c})$.
\end{proof}

\autoref{notcomeager2} shows a connection between $\Ncal_\Jcal$ and meager sets. First recall:

\begin{lemma}[See e.g.~{\cite{kechris}}]\label{ncominclp}
Let $A\subseteq {}^\omega2$. If $A$ is meager in $\cantor$ then, for any  $s\in 2^{<\omega}$, $[s]\cap A$ is not co-meager in $[s]$. The~converse is true when $A$ has the~Baire property.
\end{lemma}

Consequences of \autoref{notcomeager2} are:

\begin{corollary}\label{cor:notcomeager}
If $\Jcal$ is an ideal on $\omega$, then $\Jcal$ does not have the Baire property iff $Z$ is not co-meager in $\cantor$ for any $Z\in\Ncal_\Jcal$.
\end{corollary}
\begin{proof}
For the implication from left to right, apply \autoref{notcomeager2} to $C=\cantor$. For the converse, if $\Jcal$ has the Baire property then $\Ncal_\Jcal=\Ncal$ by \autoref{BaireForNJ2}, and it is well-known that $\Ncal$ contains a co-meager set in $\cantor$ (Rothberger's Theorem~\cite{Roth38}).
\end{proof} 

\begin{corollary}\label{Njmeag}
Assume that $\Jcal$ is an ideal on $\omega$ without the Baire property.
Then, for any 
$Z\in\Ncal_\Jcal$, $Z$ is meager in $\cantor$ iff it has the Baire property.
\end{corollary}
\begin{proof}
Assume that $Z\in \Ncal_\Jcal$ has the Baire property. By \autoref{notcomeager2}, $Z\cap[s]$ is not co-meager in $[s]$ for all $s\in 2^{<\omega}$. Therefore, by \autoref{ncominclp}, $Z$ is meager in $\cantor$.
\end{proof}

\autoref{notcomeager2} gives us the converse of \autoref{BaireForNJ2} for $\Ncal_\Jcal$.
Furthermore, we can prove that there is a~meager set of Lebesgue measure zero which is not contained in $\NidealJ{\Jcal}$ for non-meager $\Jcal$.

\begin{theorem}\label{charaktNj2} Let $\Jcal$ be an~ideal on~$\omega$.
   Then, the following statements are equivalent.
\begin{multicols}{3}
\begin{enumerate}[label=\rm (\roman*)]
       \item\label{it:NJneqN} $\Ncal_\Jcal\subsetneq\Ncal$.
    \item\label{it:MNnsubNJ} $\meager\cap\mnula \nsubseteq \NidealJ{\Jcal}$.
     \item\label{it:JnB} $\Jcal$ is not meager.
\end{enumerate}
\end{multicols}
\end{theorem}

\begin{proof} The~implication $\ref{it:NJneqN}\rightarrow\ref{it:JnB}$ follows directly from \autoref{BaireForNJ2}. On the other hand, since $\meager\cap\mnula \subseteq \mathcal{N}$ the~implication $\ref{it:MNnsubNJ}\to\ref{it:NJneqN}$ is obvious.

It remains to show $\ref{it:JnB}\to\ref{it:MNnsubNJ}$. First, we choose a closed nowhere dense $C\subseteq ^\omega2$ of positive measure, which can be found self-supported (as in the hypothesis of \autoref{notcomeager2}). 
Then, we find $G\subseteq C$ which is co-meager in $C$ and has Lebesgue measure zero. Hence $G\in\meager\cap\Ncal$, but $G\notin\Ncal_\Jcal$ by \autoref{notcomeager2}.
\end{proof}

The proof of $\ref{it:JnB}\to\ref{it:MNnsubNJ}$ is similar to the proof of $\Ecal\subsetneq\meager\cap\Ncal$ from~\cite[Lemma~2.6.1]{BJ}. Actually, the latter is already implied by $\ref{it:MNnsubNJ}$ (see \eqref{eq:contfin}).

However, $\Ncal_\Jcal$, and even $\Ncal^*_\Jcal$, contain many non-meager sets when $\Jcal$ is not meager. Examples can be obtained from the following construction.

\begin{definition}\label{def:cT}
For any non-empty tree $T\subseteq2^{<\omega}$ without maximal nodes we define $\bar{c}^T$ as follows. Enumerate $T=\{t_n:n\in\omega\}$ in such a way that $t_m\subseteq t_n$ implies $m\leq n$. By recursion on $n$, construct $\{t^n_k:k\leq n\}\subseteq T$ such that:
\begin{enumerate}[label=\rm (\roman*)]
    \item\label{it:t0} $t^0_0\in T$ extends $t_0$ and  $|t^0_0|\geq1$
    \item\label{it:tn+1} $t^{n+1}_{n+1}\supseteq t_{n+1}$,
    \item\label{it:tn+1'} for each $k\leq n$, $t^{n+1}_k\supseteq t^n_k$ and,
    \item for each $k\leq n+1$, $|t^{n+1}_k|\geq 2n+2$.
\end{enumerate}
Define $c^T_n:=\bigcup_{k\leq n}[t^n_k]$ and $\bar{c}^T:=\seqn{c^T_n}{n\in\omega}$. In addition, the~inequality $\leb{c^T_n}\leq2^{-n}$ holds, so $\bar c^T\in\omegabar$.
\end{definition}

In general, for any sequence $\bar\varepsilon=\seqn{\varepsilon_n}{n<\omega}$ of positive reals, it is possible to construct a similar $\bar{c}^{T,\bar\varepsilon}$ such that $\leb{c^{T,\bar\varepsilon}_n}<\varepsilon_n$ for all $n<\omega$, that is, $\bar{c}^{T,\bar\varepsilon}\in\Omega^*_{\bar\varepsilon}$. For this, we just need to modify the length of each $t^n_k$ accordingly.

\begin{mainlemma}\label{capcomeag2}
   Let $T\subseteq2^{<\omega}$ be a non-empty tree without maximal nodes and let $\Jcal$ be a non-meager ideal on~$\omega$. Then $N^*_\Jcal(\bar{c}^T)\cap G\neq\emptyset$ for every  co-meager set \mbox{$G\subseteq[T]$} in $[T]$, i.e.\ $N^*_\Jcal(\bar{c}^T)$ is not meager in $[T]$. 
   
   In other words, for any non-empty closed $H\subseteq\cantor$ there is some $\bar{c}^H\in\omegabar$ such that, for any non-meager ideal $\Jcal$ on $\omega$, $N^*_\Jcal(\bar c^H)\cap H$ is non-meager in $H$.
\end{mainlemma}
\begin{proof} 
Let $G\subseteq[T]$ be co-meager in $[T]$. Then,
there is a sequence $\seqn{D_n}{n\in\omega}$ of open dense sets in $[T]$ such that $\bigcap_{n\in\omega}D_n\subseteq G$.
We look for an $x\in N^*_\Jcal(\bar{c}^T)\cap\bigcap_{n\in\omega}D_n$ by using the non-meager game. 
   
Construct the following strategy of Player~I for the non-meager game for $\Jcal^d$ along with fragments of the~desired $x$.

\underline{The~first move}: 

Player~I chooses some $s_0\in T$ extending $t^0_0$ such that $[s_0]\cap[T]\subseteq D_0$. Since $s_0\in T$, $s_0=t_{n_0}$ for some $n_0\in\omega$. Player~I moves with $A_0:=n_0+1$.

\underline{Second move and further}:   

Player~II replies with $B_0$. Since $B_0\cap A_0=\emptyset$, by~\ref{it:tn+1}--\ref{it:tn+1'} of \autoref{def:cT}, Player~I can extend $s_0$ to some $s'_0\in T$ such that, for any $\ell\in B_0$, $s'_0$ extends $t^\ell_{n_0}$, 
and further finds $s_1\in T$ extending $s'_0$ such that $[s_1]\cap[T]\subseteq D_1$. Here, $s_1=t_{n_1}$ for some $n_1\in\omega$, and Player~I moves with $A_1:=n_1+1$. 
Player~II would reply with some $B_1$ not intersecting $A_1$, and Player~I continues playing in the same way.
   
Now, since $\Jcal$ is not meager, Player~I does not have a winning strategy. In particular, there is some match $\seqn{(A_n,B_n)}{n\in\omega}$ of the game where Player~I uses the strategy defined above and Player~II wins, i.e.\ $F:=\bigcup_{n\in\omega}B_n$ is in $\Jcal^d$. 
On the~other hand, Player~I constructed the increasing sequence $\seqn{s_n}{n\in\omega}$ of members of $T$, so $x:=\bigcup_{n\in\omega}s_n$ is a branch of the tree. By the definition of the strategy, we have that $x\in c^T_m$ for any $m\in F$, so $x\in N^*_\Jcal(\bar{c}^T)$. Also $x\in D_n$ for every $n\in\omega$.
\end{proof}

As a consequence for $T=2^{<\omega}$, 
we conclude that 
$\NstaridealJ{\Jcal}$ 
contains non-meager subsets of $\cantor$ when $\Jcal$ is an ideal on $\omega$ without the Baire property. 
Moreover, let us emphasize that the~same is true for $\NidealJ{\Jcal}$ as well (because $N^*_\Jcal(\bar c)\subseteq N_\Jcal(\bar c)$).

\begin{corollary}\label{cor:ML2}
  If $\Jcal$ is not meager then $N^*_\Jcal(\bar{c}^T)$ and $N_\Jcal(\bar{c}^T)$ are not meager in $^\omega2$.
\end{corollary}

In contrast with \autoref{capcomeag2}, if $\Jcal$ has the Baire property then $N^*_\Jcal(\bar c^T)$ is meager in $\cantor$ (and even $N^*_\Jcal(\bar c^T)\cap[T]$ is meager in $[T]$ when $\leb{[T]}>0$) because $N_\Jcal^*(\bar c^T)\in\Ecal$ by \autoref{BaireForNJ2}. 

As a consequence of \autoref{capcomeag2}, we can finally conclude that $\Ncal^*_\Jcal=\Ecal$ iff $\Jcal$ has the Baire property.

\begin{theorem}\label{conseqforE2}
   
   Let $\Jcal$ be an ideal on $\omega$. Then  
   the~following statements are equivalent:
   \begin{multicols}{2}
   \begin{enumerate}[label=\rm (\roman*)]
       \item\label{it:N*1}  $\Ncal^*_\Jcal\nsubseteq\meager$.
       \item\label{it:N*2}   $\Ncal_\Jcal^*\nsubseteq\meager\cap\mnula$.
       \item\label{it:N*3}   $\mathcal{E}\subsetneq\Ncal^*_\Jcal$.
       \item\label{it:N*4} $\Jcal$ is not meager.
   \end{enumerate}
   \end{multicols}
\end{theorem}

\begin{proof} 
    $\ref{it:N*4}\rightarrow\ref{it:N*1}$ follows by \autoref{cor:ML2}; $\ref{it:N*1}\rightarrow\ref{it:N*2}$ is obvious; $\ref{it:N*2}\rightarrow\ref{it:N*3}$ is a consequence of $\Ecal\subseteq\meager\cap\Ncal$; $\ref{it:N*3}\rightarrow\ref{it:N*4}$ is a consequence of \autoref{BaireForNJ2}.
\end{proof}

In the case of $\Ncal_\Jcal$, whether $\Jcal$ is meager or not, $\Ncal_\Jcal$ contains non-meager sets.

\begin{corollary}
  For every ideal $\Jcal$ on $\omega$, $\Ncal_\Jcal\nsubseteq\meager$ and  $\NidealJ{\Jcal}\nsubseteq\meager\cap\mnula$.
\end{corollary}
\begin{proof}
Since $\Ncal_\Jcal^*\subseteq\Ncal_\Jcal$, the conclusion is clear by \autoref{conseqforE2} when $\Jcal$ is not meager. Otherwise $\Ncal_\Jcal=\Ncal$, and $\Ncal$ contains a co-meager set.
\end{proof}

\autoref{relationNJ2} summarizes the situation in \autoref{charaktNj2} and~\ref{conseqforE2} when $\Jcal$ is a non-meager ideal on $\omega$.

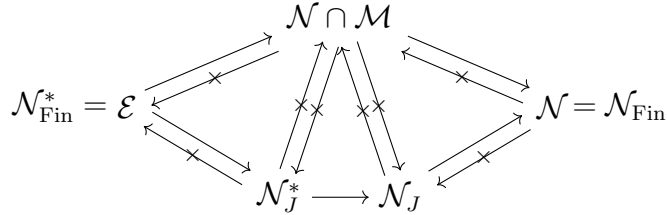
\begin{figure}[H]
\centering
\begin{tikzpicture}[scale=0.8, squarednode/.style={rectangle, draw=white, very thick, minimum size=7mm}]
\node (a) at (-1, -1.5) {$\NstaridealJ{\Jcal}$};
\node (f) at (1, -1.5) {$\NidealJ{\Jcal}$};
\node (bb) at (-4.6, 0) {$\NstaridealJ{\fin}=$};
\node (b) at (-3.5, 0) {$\mathcal{E}$};
\node (c) at (3.5, 0) {$\mathcal{N}$};
\node (cc) at (4.6, 0) {$=\NidealJ{\fin}$};
\node[squarednode] (d) at (0, 1.5) {$\mathcal{N}\cap\mathcal{M}$};


\draw[->] (a) edge (f);

\foreach \from/\to in {b/a, b/d, d/c, f/c}
{
\path[->,shift left=0.75ex] (\from) edge (\to);
\path[->,shift left=0.75ex] (\to) edge node{\scriptsize $\times$} (\from);
}

\foreach \from/\to in {a/d, f/d}
{
\path[->,shift left=0.75ex] (\from) edge node{\scriptsize $\times$} (\to);
\path[->,shift left=0.75ex] (\to) edge node{\scriptsize $\times$} (\from);
}
\end{tikzpicture}

\caption{The situation when $\Jcal$ is an ideal on $\omega$ without the Baire property. An arrow denotes $\subseteq$, while a crossed arrow denotes $\nsubseteq$. The arrow on the bottom could be reversed, e.g.\ when $\Jcal$ is a maximal ideal.}
\label{relationNJ2}
\end{figure}


\section{The effect of nearly coherence of filters}\label{sec:near}

In this section, we prove a characterization of nearly coherence of filters (or ideals) in terms of the ideals $\Ncal_\Jcal$ and $\Ncal^*_\Jcal$. We first recall the notion of nearly coherence.

\begin{definition}[A.~Blass~{\cite{blass86}}]\label{def:NC}
Two filters $\Fcal_0$ and $\Fcal_1$ on~$\omega$ are \emph{nearly coherent} if there is a~finite-to-one function $\varphi\in{}^\omega\omega$ such that $\varphi^\to(\Fcal_0)\cup \varphi^\to(\Fcal_1)$ has the~finite intersection property. Dually, we say that two ideals $\Jcal_0$ and $\Jcal_1$ are \emph{nearly coherent} if there is a finite-to-one function $\varphi\in{}^\omega\omega$ such that $\varphi^\to(\Jcal_0)\cup\varphi^\to(\Jcal_1)$ is contained in some ideal.
\end{definition}

If $J_0$ and $J_1$ are nearly coherent ideals on $\omega$, then there is some ideal $K$ in $\omega$ which is $\leqKBpr$-above both $J_0$ and $J_1$.
As a consequence of \autoref{RBeq}~\ref{it:underKB'}:

\begin{lemma}\label{lem:coh}
  If $\Jcal_0$ and $\Jcal_1$ are nearly coherent ideals on $\omega$, then there is some ideal $\Kcal$ on $\omega$ such that $\Ncal^*_{\Jcal_0}\cup\Ncal^*_{\Jcal_1} \subseteq \Ncal^*_\Kcal\subseteq\Ncal_\Kcal \subseteq \Ncal_{\Jcal_0}\cap\Ncal_{\Jcal_1}$ (see \autoref{fig:coherence}).
\end{lemma}

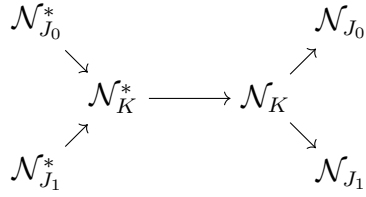
\begin{figure}[H]
\centering
\begin{tikzpicture}
\node (N*I) at (-1,2) {$\Ncal^*_{\Jcal_0}$};
\node (N*J) at (-1,0) {$\Ncal^*_{\Jcal_1}$};
\node (N*K) at (0,1) {$\Ncal^*_\Kcal$};
\node (NK) at (2,1) {$\Ncal_\Kcal$};
\node (NJ) at (3,0) {$\Ncal_{\Jcal_1}$};
\node (NI) at (3,2) {$\Ncal_{\Jcal_0}$};

\draw (N*I) edge[->] (N*K)
      (N*J) edge[->] (N*K)
      (N*K) edge[->] (NK)
      (NK)  edge[->] (NI)
      (NK)  edge [->] (NJ);
\end{tikzpicture}
\caption{Situation describing \autoref{lem:coh} (an arrow denotes $\subseteq$).}\label{fig:coherence}
\end{figure}

Since $\Ncal^*_\Jcal=\Ncal_\Jcal$ for any maximal ideal $\Jcal$ on $\omega$, we immediately obtain:

\begin{corollary}\label{cor:NCuf}
  If $\Jcal$ and $\Kcal$ are nearly coherent ideals and $\Kcal$ is \underline{maximal}, then $\Ncal^*_\Jcal\subseteq\Ncal_\Kcal\subseteq\Ncal_\Jcal$. In particular, if $\Jcal$ is also a \underline{maximal} ideal, then $\Ncal_\Jcal=\Ncal_\Kcal$.
\end{corollary}

A.~Blass~\cite{blass86} has introduced the following principle, which was proved consistent with ZFC by Blass and Shelah~\cite{blassshelah}.

\begin{enumerate}[label=\rm (NCF)]
    \item\label{NCF} \emph{Near coherence of filters:} Any pair of filters on $\omega$ are nearly coherent.
\end{enumerate}

In fact, $\mathfrak{u}<\mathfrak{g}$ (which holds in Miller model~\cite{BlassSP,blassshelah3}) implies \ref{NCF} (\cite[Cor.~9.18]{blassbook}). On the other hand, it is possible to obtain not nearly coherence pairs of filters under CH and in random model~\cite[Sec.~4]{blass86}, as well as in Cohen model (e.g.~\autoref{cohenthm}).

As a consequence of \autoref{cor:NCuf}:

\begin{corollary}\label{cor:NCFNuf}
  \ref{NCF} implies that all $\Ncal_\Jcal$ with $\Jcal$ maximal ideal on $\omega$ are the same.
\end{corollary}

We prove the converse of \autoref{lem:coh} and of \autoref{cor:NCuf} and~\ref{cor:NCFNuf}, which gives us a characterization of nearly coherence and~\ref{NCF}. For this purpose, we use the following game, formulated by
T.~Eisworth, that characterizes nearly coherence. 

\begin{definition}[Eisworth~{\cite{Eis01}}]
Let $\Fcal_0$, $\Fcal_1$ be two filters on~$\omega$. The following game of length $\omega$ between two players is called the \emph{nearly coherence game $C_{\Fcal_0,\Fcal_1}$}:
\begin{itemize}
   \item In the $n$-th move, Player~I plays a~finite set $A_n\in[\omega]^{<\omega}$ and 
   Player~II responds with a~finite set $B_n\in[\omega]^{<\omega}$ disjoint from $A_n$. 
   \item After $\omega$ many moves, Player~II wins if $\bigcup\set{B_{2n+i}}{n\in\omega}\in\Fcal_i$ for $i\in\{0,1\}$. Otherwise,
   Player~I wins.
\end{itemize}
\end{definition}

\begin{theorem}[Eisworth~{\cite{Eis01}}]
Let $\Fcal_0$, $\Fcal_1$ be two filters on~$\omega$. Then $\Fcal_0$ and $\Fcal_1$ are nearly coherent iff Player~I has a winning strategy of the game $C_{\Fcal_0,\Fcal_1}$.
\end{theorem}

We use the~nearly coherence game $C_{\Fcal_0,\Fcal_1}$ to prove the~following technical lemma.

\begin{mainlemma}\label{notnear2}
     Let $C\subseteq {}^\omega2$ be a closed self-supported set. 
   If $\Jcal$ and $\Kcal$ are not nearly coherent ideals on $\omega$ then $N^*_\Jcal(\bar{c}^T)\notin\Ncal_\Kcal$ where $T$ is the~tree without maximal nodes such that $C=[T]$, and $\bar c^T$ is as in \autoref{def:cT}.
\end{mainlemma}
\begin{proof}
We use the nearly coherence game $C_{\Jcal^d,\Kcal^d}$ and build a strategy for Player~I, in order to get an~\mbox{$x\in N^*_\Jcal(\bar{c}^T)\smallsetminus N_\Kcal(\bar{c}')$} for any $\bar{c}'\in\omegabar$.

\underline{The~first move}: 
Player~I first moves with $A_0=\{0\}$, and put $s_0:= t_0$ (which is the~empty sequence).

\underline{The second and third moves, and further}:

After Player~II replies with $B_0\subseteq \omega\smallsetminus A_0$, Player~I finds $s_1\supseteq s_0$ in $T$ such that $s_1\supseteq t^n_0$ for every $n\in B_0$,  
and further finds $n_1\geq|s_1|$ such that $\mu\left(\bigcup_{\ell\geq n_1}c'_\ell\right)<\mu(C\cap[s_1])$, which implies that there is some $x_1\in C\cap [s_1]\smallsetminus\bigcup_{\ell\geq n_1}c'_\ell$.  Player~I moves with $A_1=n_1+1$. 

After Player~II replies with $B_1\subseteq \omega\smallsetminus A_1$, Player~I finds some $m_2\geq n_1$ such that $[s_2]\cap\bigcup_{n\in B_1}c'_n=\emptyset$ where $s_2:=x_1{\upharpoonright}m_2$. Player~I moves $A_2=n_2+1$ where $n_2<\omega$ is such that $s_2=t_{n_2}$. 

Afterwards, Player~II moves with $B_2\subseteq\omega\smallsetminus A_2$, and the same dynamic is repeated: Player~I finds an $s_3\supseteq s_2$ in $T$ such that $s_3\supseteq t^n_{n_2}$ for every $n\in B_2$, 
and some $n_3\geq|s_3|$ such that $\mu\left(
\bigcup_{\ell\geq n_3} c'_\ell\right)<\mu(C\cap[s_3])$, which implies that there is an $x_3\in C\cap [s_3]\smallsetminus\bigcup_{\ell\geq n_3}c'_\ell$. 
Player~I moves with $A_3=n_3+1$, and the game continues as described so far.

   Since Player~I does not have a winning strategy, there is some run as described above where Player~II wins. Thus, $F_0:=\bigcup_{n\in\omega}B_{2n}\in\Jcal^d$ and $F_1:=\bigcup_{n\in\omega}B_{2n+1}\in\Kcal^d$. Set $x:=\bigcup_{n\in\omega}s_n$. Note that $x\in c^T_\ell$ for any $\ell\in F_0$, and $x\notin c'_\ell$ for any $\ell\in F_1$, which means that $x\in N^*_\Jcal(\bar{c}^T)$ and $x\notin N_\Kcal(\bar{c}')$.
\end{proof}

An application of the previous result to $C=\cantor$ yields:

\begin{theorem}\label{thm:nNCNJ}
   If $\Jcal$ and $\Kcal$ are not near-coherent ideals on $\omega$ then $\Ncal^*_\Jcal\nsubseteq\Ncal_\Kcal$ and $\Ncal^*_\Kcal\nsubseteq\Ncal_\Jcal$.
   In particular, $\NidealJ{\Kcal}\neq \NidealJ{\Jcal}$ and $\Ncal^*_\Jcal\neq\Ncal^*_\Kcal$.
\end{theorem}



It is clear that $\Fin$ is nearly coherent with any filter on $\omega$. Therefore, any pair of not nearly coherent filters must be non-meager. As a consequence of \autoref{charaktNj2} and~\ref{conseqforE2}:

\begin{corollary}\label{cor:nnc-NE}
   Let $\Jcal$ and $\Kcal$ be not nearly coherent ideals on~$\omega$. Then the ideals
$\mathcal{E}$, $\NstaridealJ{\Jcal}$, $\NidealJ{\Kcal}$ and $\mathcal{N}$ are pairwise different.
\end{corollary}

The situation in \autoref{thm:nNCNJ} and \autoref{cor:nnc-NE} is illustrated in  \autoref{nearcoherence}.

\begin{figure}[H]
\centering
\begin{tikzpicture}
\node (a) at (-2, -1.5) {$\NstaridealJ{\Jcal}$};
\node (f) at (2, -1.5) {$\NidealJ{\Jcal}$};
\node (b) at (-3.5, 0) {$\mathcal{E}$};
\node (c) at (3.5, 0) {$\mathcal{N}$};
\node (d) at (-2, 1.5) {$\NstaridealJ{\Kcal}$};
\node (ff) at (2, 1.5) {$\NidealJ{\Kcal}$};

\foreach \from/\to in { a/f, d/ff}
\draw [->] (\from) -- (\to);

\foreach \from/\to in { a/ff, d/f, ff/a, f/d, a/d, f/ff, d/a, ff/f}
{
\path[-,shift left=0.75ex,draw=white,line width=3pt] (\from) edge (\to);
\path[->,shift left=0.75ex] (\from) edge node[near start]{\scriptsize /} (\to);
}

\foreach \from/\to in {b/a, b/d, f/c, ff/c}
{
\path[->,shift left=0.75ex] (\from) edge (\to);
\path[->,shift left=0.75ex] (\to) edge node{\scriptsize /} (\from);
}
\end{tikzpicture}

\caption{Diagram corresponding to the situation in \autoref{thm:nNCNJ} where $\Jcal$ and $\Kcal$ are not nearly coherent ideals on $\omega$. An arrow denotes $\subseteq$, and a crossed arrow denotes $\nsubseteq$. The arrow on the top could be reversed, e.g.\ when $\Kcal$ is a maximal ideal (likewise for the arrow on the bottom).}
\label{nearcoherence}
\end{figure}
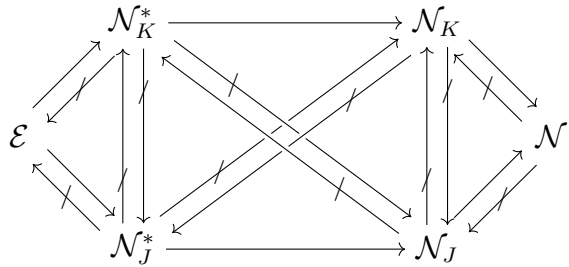

We summarize our results as a characterization of nearly coherence.

\begin{theorem}\label{nc-char}
Let $J$ and $K$ be ideals on $\omega$. The following statements are equivalent.
\begin{enumerate}[label =\rm (\roman*)]
    \item\label{it:nc} $J$ and $K$ are nearly coherent.
    \item\label{it:snc}
    There is some ideal $K'$ on $\omega$ such that $J \leqKBpr K'$ and $K\leqKBpr K'$.
    \item\label{it:aida} There is some ideal $K'$ on $\omega$ such that   $\Ncal^*_J\cup \Ncal^*_{K}\subseteq \Ncal^*_{K'}\subseteq \Ncal_{K'}\subseteq \Ncal_J \cap \Ncal_K$.
    \item\label{it:subset} $\Ncal^*_J\subseteq \Ncal_K$.
\end{enumerate}
\end{theorem}
\begin{proof}
$\ref{it:nc}\imp \ref{it:snc}$ is obvious; $\ref{it:snc}\imp \ref{it:aida}$ is immediate from \autoref{RBeq}; $\ref{it:aida}\imp\ref{it:subset}$ is obvious; and $\ref{it:subset}\imp \ref{it:nc}$ follows by (the contrapositive of) \autoref{thm:nNCNJ}.
\end{proof}

The non-nearly coherence of filters also gives us examples of non-meager ideals $J$ such that $\Ncal^*_J\neq\Ncal_J$. We do not know how to construct such an example in ZFC.

\begin{lemma}\label{noncoh:sum}
  Assume that $J_1$ and $J_2$ are non-nearly coherent ideals on $\omega$. Then $J_1\oplus J_2$ is non-meager and $\Ncal^*_{J_1\oplus J_2}\neq \Ncal_{J_1\oplus J_2}$.
\end{lemma}
\begin{proof}
Partition $\omega=\N_1\cup \N_2$ into two infinite sets, let $g_e\colon \omega\to\N_e$ be a bijection for each $e\in\{1,2\}$, and let $J'_e= g_e^\to(J_e)$, which is an ideal on $\N_e$ isomorphic with $J_e$. Here, our interpretation of $J_1\oplus J_2$ is $J'_1\oplus J'_2$. It is clear that $\Ncal_{J'_1}=\Ncal_{J_1}$ and $\Ncal^*_{J'_1}=\Ncal^*_{J_1}$ (by \autoref{RBeq}~\ref{it:underRB}).

Recall that non-nearly coherent ideals must be non-meager (this also follows by \autoref{thm:nNCNJ}), 
so $J'_1\times J'_2$ is a non-meager subset of $\cantor= {}^{\N_1}2 \times {}^{\N_2}2$ by Kuratowski-Ulam Theorem, which implies that $J'_1\oplus J'_2$ is non-meager.

By \autoref{exm:sum}, $\Ncal^*_{J'_1\oplus J'_2}=\Ncal^*_{J'_1}\cap \Ncal^*_{J'_2} = \Ncal^*_{J_1}\cap \Ncal^*_{J_2}$ and $\Ncal_{J'_1 \oplus J'_2}$ contains both $\Ncal_{J_1}$ and $\Ncal_{J_2}$, so $\Ncal^*_{J'_1\oplus J'_2}= \Ncal_{J'_1\oplus J'_2}$ would imply that $\Ncal_{J_1}\cup \Ncal_{J_2}\subseteq \Ncal^*_{J_1}\cap\Ncal^*_{J_2}$, and in turn $\Ncal_{J_2}\subseteq \Ncal^*_{J_1}$, which implies $\Ncal^*_{J_2}\subseteq \Ncal_{J_1}$ (because $\Ncal^*_{J_e}\subseteq \Ncal_{J_e}$), contradicting \autoref{nc-char} and the fact that $J_1$ and $J_2$ are not nearly-coherent. Therefore, $\Ncal^*_{J'_1\oplus J'_2}\neq \Ncal_{J'_1\oplus J'_2}$
%
%
\end{proof}

In contrast with the previous result, we do not know whether \ref{NCF} implies that all $\Ncal_J$ are the same for non-meager $J$.

\section{Cardinal characteristics}\label{sec:cardinalityNJ}

In this section, we focus on investigating the cardinal characteristics associated with $\Ncal_J$ and $\Ncal^*_J$.

We review some basic notation about cardinal characteristics. Many cardinal characteristics are defined using relational systems in the following way~\cite{Vojtas}. A \emph{relational system} is a triplet $\Rbf=\la X,Y,R\ra$ where $R$ is a relation and $X$ and $Y$ are non-empty sets.\footnote{It is typically assumed that $R\subseteq X\times Y$, but it is not required. In fact, $R$ could be a proper class relation like $\subseteq$ and $\in$.} Define
  \begin{linenomath}
  \begin{align*}
        \bfrak(\Rbf) & :=\min\set{|F|}{F\subseteq X,\ \neg\,(\exists\, y\in Y)\ (\forall\, x\in F)\  x\, R\, y}\\
        \dfrak(\Rbf) & :=\min\set{|D|}{D\subseteq Y,\ (\forall\, x\in X)\ (\exists\, y\in D)\ x\, R\, y}.
  \end{align*}
  \end{linenomath}
  
The \emph{dual of $\Rbf$} is defined by $\Rbf^\perp:=\la Y,X,R^\perp\ra$ where $y\, R^\perp\, x$ iff $\neg\, x\, R\, y$. Hence $\bfrak(\Rbf^\perp)=\dfrak(\Rbf)$ and $\dfrak(\Rbf^\perp)=\bfrak(\Rbf)$.
  
Given another relational system $\Rbf'=\la X', Y', R'\ra$, say that a pair $(\varphi_-,\varphi_+)$ is a \emph{Tukey connection from $\Rbf$ to $\Rbf'$} if $\varphi_-:X\to X'$, $\varphi_+:Y'\to Y$ and, for any $x\in X$ and $y'\in Y'$, if $\varphi_-(x)\, R'\, y'$ then $x\, R\, \varphi_+(y')$.

We say that \emph{$\Rbf$ is Tukey below $\Rbf'$}, denoted by $\Rbf\leqT\Rbf'$, if there is a Tukey connection from $\Rbf$ to $\Rbf'$. Say that $\Rbf$ is \emph{Tukey equivalent} to $\Rbf'$, denoted by $\Rbf\eqT\Rbf'$, if $\Rbf\leqT\Rbf'$ and $\Rbf'\leqT\Rbf$. It is known that $\Rbf\leqT\Rbf'$ implies $\dfrak(\Rbf)\leq\dfrak(\Rbf')$ and $\bfrak(\Rbf')\leq\bfrak(\Rbf)$. Hence $\Rbf\eqT\Rbf'$ implies $\dfrak(\Rbf)=\dfrak(\Rbf')$ and $\bfrak(\Rbf')=\bfrak(\Rbf)$.

We constantly use the relational system $\cbf^\Jcl_\Icl:=\la\Icl,\Jcl,\subseteq\ra$ discussed in~\cite[Ch.~2]{BJ}, and we identify $\Icl$ with the relational system $\cbf^\Icl_\Icl$. Denote $\add(\Icl,\Jcl):=\bfrak(\cbf^\Jcl_\Icl)$ and $\cof(\Icl,\Jcl):=\dfrak(\cbf^\Jcl_\Icl)$. These cardinal characteristics are interesting when $\Icl\subseteq\Jcl$ are ideals on some set $X$. It is well-known that, for any ideal $\Icl$ on $X$, we can express the \emph{cardinal characteristics associated with $\Icl$} as follows:
\begin{linenomath}
\begin{align*}
    \add(\Icl)& =\add(\Icl,\Icl), &  \cof(\Icl) &=\cof(\Icl,\Icl),\\
    \non(\Icl)& =\add([X]^{<\aleph_0},\Icl), & \cov(\Icl)& =\cof([X]^{<\aleph_0},\Icl).
\end{align*}
\end{linenomath}
In fact, via the relational system $\Cbf_\Icl:=\la X,\Icl,\in\ra$, we obtain $\non(\Icl)=\bfrak(\Cbf_\Icl)$ and $\cov(\Icl)=\dfrak(\Cbf_\Icl)$. The following easy claims illustrate basic relations between these cardinal characteristics.

\begin{fact}\label{wcov}
If $\Icl$ is an ideal on $X$, $\Icl\subseteq\Icl'$ and $\Jcl\subseteq\pts(X)\smallsetminus \{X\}$, then $(\cbf_{\Icl'}^\Jcl)^\perp \leqT \Cbf_\Icl\leqT\cbf_{[X]^{<\aleph_0}}^\Icl$. In particular, $\add(\Icl',\Jcl)\leq\cov(\Icl)$ and $\non(\Icl)\leq\cof(\Icl',\Jcl)$.
\end{fact}
\begin{proof}
We only show the first Tukey connection. Define $F\colon \Jcl\to X$ such that $F(B)\in X\smallsetminus B$ (which exists because $X\notin \Jcl$), and define $G\colon \Icl\to \Icl'$ by $G(A):=A$. Then, for $A\in\Icl$ and $B\in\Jcl$, $F(B)\in A$ implies $B\nsupseteq A$. Hence, $(F,G)$ witnesses $(\cbf_{\Icl'}^\Jcl)^\perp \leqT \Cbf_\Icl$.
\end{proof}

\begin{fact}\label{cofIJrel}
If $\Icl\subseteq\Icl'$ and $\Jcl'\subseteq \Jcl$ then $\cbf^\Jcl_\Icl\leqT \cbf^{\Jcl'}_{\Icl'}$. In particular, $\add(\Icl',\Jcl')\leq \add(\Icl,\Jcl)$ and $\cof(\Icl,\Jcl)\leq \cof(\Icl',\Jcl')$.
\end{fact}

\begin{corollary}\label{corcovsub}
  If $\Jcl'\subseteq\Jcl$ are ideals on $X$, then $\cbf^{\Jcl}_{[X]^{<\aleph_0}} \leqT \cbf^{\Jcl'}_{[X]^{<\aleph_0}}$ and $\Cbf_{\Jcl}\leqT \Cbf_{\Jcl'}$. In particular, $\cov(\Jcl)\leq \cov(\Jcl')$ and $\non(\Jcl')\leq\non(\Jcl)$.
\end{corollary}

We look at the cardinal characteristics associated with $\Ncal_J$ and $\Ncal^*_J$ when $J$ is an ideal on $\omega$. If $\Jcal$ has the~Baire property then
the cardinal characteristics associated with $\mathcal{N}_\Jcal$ equal to those associated with $\mathcal{N}$ because $\Ncal_J=\Ncal$ (\autoref{BaireForNJ2}). Moreover, since $\Ncal^*_J=\Ecal$, the~cardinal characteristics associated with $\NstaridealJ{\Jcal}$ equal to those associated with $\mathcal{E}$. We recall below some results about the cardinal characteristics associated with $\Ecal$. 

\begin{theorem}[{\cite{BartSh}}, see also~{\cite[Sec.~2.6]{BJ}}]\label{thm:E}
\ 
  \begin{enumerate}[label =\rm (\alph*)]
      \item $\min\{\bfrak,\non(\Ncal)\} \leq\non(\Ecal)\leq\min\{\non(\Mcal),\non(\Ncal)\}$.
      \item $\max\{\cov(\Mcal),\cov(\Ncal)\}\leq\cov(\Ecal)  \leq\max\{\dfrak,\cov(\Ncal)\}$.
      \item $\add(\Ecal,\Ncal)=\cov(\meager)$ and $\cof(\Ecal,\Ncal)=\non(\meager)$.
      \item $\add(\Ecal)=\add(\Mcal)$ and $\cof(\Ecal)=\cof(\Mcal)$.
  \end{enumerate}
\end{theorem}

\begin{theorem}[{\cite[Lem.~7.4.3]{BJ}}]
  $\sfrak\leq \non(\Ecal)$ and $\cov(\Ecal)\leq \rfrak$.
\end{theorem}

Thanks to the previous results, and the fact that $\Ecal\subseteq\NstaridealJ{\Jcal}\subseteq\NidealJ{\Jcal}\subseteq\mathcal{N}$, we obtain:

\begin{theorem}\label{covnonvscichon}
  $\mathrm{ZFC}$ proves
  \[\begin{array}{ccccccccc}
      \cov(\Ncal) &\leq & \cov(\Ncal_J) & \leq & \cov(\Ncal^*_J) &\leq & \cov(\Ecal) &\leq & \min\{\cof(\Mcal),\rfrak\},\\[1ex] 
      \max\{\add(\Mcal), \sfrak\} &\leq & \non(\Ecal) &\leq & \non(\Ncal^*_J) &\leq & \non(\NidealJ{\Jcal}) &\leq & \non(\mathcal{N}).
  \end{array}\]
\end{theorem}

We now turn to the additivity and cofinality numbers. In the case of $J=\Fin$, we can characterize $\add(\Ncal)$ and $\cof(\Ncal)$ using slaloms.

\begin{definition}
Let $b=\la b(n):\, n<\omega\ra$ be a sequence of non-empty sets, and let $h\in\baire$. Denote
\[\begin{split}
    \prod b & := \prod_{n<\omega}b(n),\\
    \Scal(b,h) & := \prod_{n<\omega}[b(n)]^{\leq h(n)}.
\end{split}\]
Define the relational system $\Lc(b,h):=\la \prod b, \Scal(b,h), \in^*\ra$ where
\[x\in ^* y \text{ iff }\set{n<\omega}{x(n)\notin y(n)} \text{ is finite}.\]
Denote $\blc_{b,h}:=\bfrak(\Lc(b,h))$ and $\dlc_{b,h}:=\dfrak(\Lc(b,h))$.

When $b$ is the constant sequence $\omega$, we use the notation $\Lc(\omega,h)$ and denote its associated cardinal characteristics by $\blc_{\omega,h}$ and $\dlc_{\omega,h}$.
\end{definition}

\begin{theorem}[Bartoszy\'nski~{\cite{BA84,BartInv}}]\label{thmBart}
Assume that $h\in\baire$ diverges to infinity. Then, $\Lc(b,h)\eqT \Ncal$. In particular, $\blc_{\omega,h}=\add(\Ncal)$ and $\dlc_{\omega,h}=\cof(\Ncal)$.
\end{theorem}

We propose the following relational system, which is practical to find bounds for the additivity and cofinality of $\Ncal_J$ and $\Ncal^*_J$.

\begin{definition}\label{def:Sbf}
Let $J$ be an ideal on $\omega$. For $c,d\in\omegabar$, define the relation
\[c\subseteq^J d \text{ iff }\set{n<\omega}{c_n\nsubseteq d_n}\in J.\]
Define the relational system $\Sbf_J:=\la \omegabar, \omegabar, \subseteq^J\ra$, and denote $\bfrak_J(\omegabar):=\bfrak(\Sbf_J)$ and $\dfrak_J(\omegabar):=\dfrak(\Sbf_J)$.
\end{definition}

It is clear that $\bfrak_J(\omegabar)$ is regular and $\bfrak_J(\omegabar)\leq\cf(\dfrak_J(\omegabar))\leq\dfrak_J(\omegabar)$.

\begin{theorem}\label{omegabarvsNJ}
Let $J$ be an ideal on $\omega$. Then $\Ncal_J\leqT \Sbf_J$ and $\Ncal^*_J\leqT \Sbf_J$. In particular, the additivities and cofinalities of $\Ncal_J$ and $\Ncal^*_J$ are between $\bfrak_J(\omegabar)$ and $\dfrak_J(\omegabar)$ (see \autoref{fig:bomegabar}).
\end{theorem}

\begin{figure}[h!]
\centering
\begin{tikzpicture}
\small{
\node (bI) at (-2,0) {$\bfrak_J(\omegabar)$};
\node (adN*I) at (0,-1) {$\add(\Ncal^*_J)$};
\node (cfN*I) at (2,-1) {$\cof(\Ncal^*_J)$};
\node (adNI) at (0,1) {$\add(\Ncal_J)$};
\node (cfNI) at (2,1) {$\cof(\Ncal_J)$};
\node (dI) at (4,0) {$\dfrak_J(\omegabar)$};

\draw (bI) edge[->] (adN*I)
      (bI) edge[->] (adNI)
      (adN*I) edge[->] (cfN*I)
      (adNI)  edge[->] (cfNI)
      (cfNI)  edge [->] (dI)
      (cfN*I)  edge [->] (dI);
}
\end{tikzpicture}
\caption{Diagram of inequalities between the cardinal characteristics associated with $\Sbf_J$, and the additivities and cofinalities of $\Ncal_J$ and $\Ncal^*_J$.}\label{fig:bomegabar}
\end{figure}
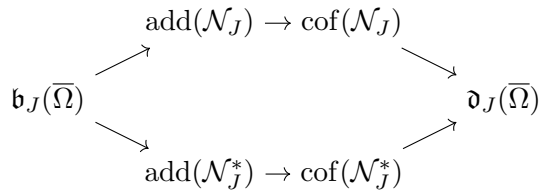

\begin{proof}
Define $F\colon \Ncal_J\to \omegabar$ such that $X\subseteq N_J(F(X))$ for any $X\in\Ncal_J$, and define $G\colon \omegabar \to \Ncal_J$ by $G(\bar d):=N_J(\bar d)$ for any $\bar d\in \omegabar$. It is clear that $(F,G)$ is a Tukey connection from $\Ncal_J$ into $\Sbf_J$, because $\bar c\subseteq^J \bar d$ implies $N_J(\bar c)\subseteq N_J(\bar d)$.

Similarly, we obtain a Tukey connection from $\Ncal^*_J$ into $\Sbf_J$ via functions $F^*\colon \Ncal^*_J\to \omegabar$ and $G^*\colon \omegabar \to \Ncal_J$ such that $X\subseteq N^*_J(F^*(X))$ for $X\in\Ncal^*_J$ and $G^*(\bar d):= N^*_J(\bar d)$.
\end{proof}

Analogous to \autoref{RBeq}, we have the following result about $\Sbf_J$.

\begin{theorem}\label{KBomegabar}
Let $J$ and $K$ be ideals on $\omega$.
\begin{enumerate}[label= \rm (\alph*)]
    \item\label{it:KBomegabar} If $K\leqKB J$ then $\Sbf_J\leqT \Sbf_K$, in particular $\bfrak_K(\omegabar)\leq\bfrak_J(\omegabar)$ and $\dfrak_J(\omegabar)\leq \dfrak_K(\omegabar)$.
    
    \item\label{it:KB'omegabar} If $K \leqKBpr J$ then $\Sbf_J\leqT \Sbf_K$, in particular $\bfrak_K(\omegabar)\leq\bfrak_J(\omegabar)$ and $\dfrak_J(\omegabar)\leq \dfrak_K(\omegabar)$.
    
    \item\label{it:RBomegabar} If $K \leqRB J$ then $\Sbf_K\eqT \Sbf_J$, in particular $\bfrak_K(\omegabar) =\bfrak_J(\omegabar)$ and $\dfrak_K(\omegabar) = \dfrak_J(\omegabar)$.
\end{enumerate}
\end{theorem}
\begin{proof}
It is clear that~\ref{it:RBomegabar} follows from~\ref{it:KBomegabar} and~\ref{it:KB'omegabar}.

Let $f\colon\omega\to \omega$ be a finite-to-one function and denote $I_n:=f^{-1}[\{n\}]$. Like in the proof of \autoref{RBeq}, define $F'\colon \omegabar\to \omegabar$ by $F'(\bar c):=\bar c'$ where $c'_n:=\bigcup_{k\in I_n}c_k$, and define $F^-\colon \omegabar\to \omegabar$ by $F^-(\bar c):=\bar c^-$ where $c^-_k:=c_{f(k)}$.

If $K\subseteq f^{\to}(J)$ then $(F',F^-)$ is a Tukey connection from $\Sbf_J$ into $\Sbf_K$, which shows~\ref{it:KBomegabar}. To prove this, assume $\bar c,\bar d\in\omegabar$ and $\bar c'\subseteq^K \bar d$, and we show $\bar c \subseteq^J \bar d^-$. The hypothesis indicates that $\set{n<\omega}{c'_n\subseteq d_n}\in K^d$, which implies that $\set{k<\omega}{c'_{f(k)}\subseteq d_{f(k)}}\in J^d$. Since $c_k\subseteq c'_{f(k)}$, the previous set is contained in $\set{k<\omega}{c_k\subseteq d^-_k}$, so $\bar c \subseteq^J \bar d^-$.

To show~\ref{it:KB'omegabar}, we verify that, whenever $f^\to(K)\subseteq J$, $(F^-,F')$ is a Tukey connection from $\Sbf_J$ into $\Sbf_K$. Let $\bar c,\bar d\in\omegabar$ and assume that $\bar c^-\subseteq^K \bar d$, i.e.\ $\set{k<\omega}{c_{f(k)}\subseteq d_k}\in K^d$. Since $d_k\subseteq d'_{f(k)}$, this set is contained in $\set{k<\omega}{c_{f(k)}\subseteq d'_{f(k)}}$, so $\set{n<\omega}{c_n\subseteq d'_n}\in J^d$, i.e.\ $\bar c\subseteq^J \bar d'$.
\end{proof}

In the case $J=\Fin$, we obtain the following characterization of the additivity and cofinality of $\Ncal$.

\begin{theorem}\label{omegabarFin}
 $\bfrak_\Fin(\omegabar)=\add(\Ncal)$ and $\dfrak_\Fin(\omegabar)=\cof(\Ncal)$.
\end{theorem}
\begin{proof}
Note that $\Ncal=\Ncal_\Fin\leqT \Sbf_\Fin$ by \autoref{omegabarvsNJ}, so $\bfrak_\Fin(\omegabar)\leq\add(\Ncal)$ and $\cof(\Ncal)\leq \dfrak_\Fin(\omegabar)$.

We show the converse inequality for $\add(\Ncal)$. It is enough to prove that, whenever $F\subseteq \omegabar$ has size ${<}\add(\Ncal)$, it has some upper $\subseteq^\Fin$-bound. For each $\bar c\in F$, find a function $f^{\bar c}\in \baire$ such that $\mu\big(\bigcup_{k\geq f^{\bar c}(n)}c_k\big)<\frac{1}{(n+1)2^n}$ for all $n<\omega$. Now $|F|<\add(\Ncal)\leq\bfrak$, so there is some increasing $f\in\baire$ with $f(0)=0$ dominating $\set{f^{\bar c}}{\bar c\in F}$, which means that, for any $\bar c\in F$, $\mu\big(\bigcup_{k\geq f(n)}c_k\big)<\frac{1}{(n+1)2^n}$ for all but finitely many $n<\omega$. By making finitely many modifications to each $\bar c\in F$, we can assume that the previous inequality is valid for all $n<\omega$. 

For each $n<\omega$, let $I_n:=[f(n),f(n+1))$.
Consider the functions $b^f$ and $h$ with domain $\omega$ such that $b^f(n):=\set{s\in{}^{I_n}\Omega}{\mu\left( \bigcup_{k\in I_n}s_k\right)<\frac{1}{(n+1)2^n}}$ and $h(n):=n+1$. Since $\Lc(b^f,h)\eqT \Lc(\omega,h)\eqT \Ncal$, we obtain that $\blc_{b^f,h}=\add(\Ncal)$ by \autoref{thmBart}. Since $F$ can be seen as a subset of $\prod b^f$, there is some $\varphi\in \Scal(b^f,h)$ such that, for any $\bar c\in F$, $\bar c\frestr I_n\in \varphi(n)$ for all but finitely many $n<\omega$. Define $\bar d\in{}^\omega \Omega$ by $d_k:=\bigcup_{s\in \varphi(n)}s_k$ for $k\in I_n$. Note that
\[\mu\left(\bigcup_{k\in I_n}d_k\right) = \mu\left( \bigcup_{s\in\varphi(n)} \bigcup_{k\in I_n}s_k\right)<\frac{1}{(n+1)2^n}(n+1)=\frac{1}{2^n},\]
thus $\bar d\in\omegabar$. On the other hand, for any $\bar c\in F$, $c_k\subseteq d_k$ for all but finitely many $k<\omega$, i.e.\ $\bar c\subseteq^\Fin \bar d$.

The proof of $\dfrak_\Fin(\omegabar)\leq \cof(\Ncal)$ is similar. Fix a dominating family $D$ of size $\dfrak$ formed by increasing functions $f$ such that $f(0)=0$. For each $f\in D$, note that $\dlc_{b^f,h}=\cof(\Ncal)$ by \autoref{thmBart}, so we can choose some witness $S^f\subseteq \Scal(b^f,h)$ and, for each $\varphi\in S^f$, define $d^{f,\varphi}_k:=\bigcup_{s\in \varphi(n)}s_k$ for $k\in[f(n),f(n+1))$. Then, $E:=\set{\bar d^{f,\varphi}}{\varphi^f\in S^f,\ f\in D}$ has size ${\leq}\cof(\Ncal)$ and it is $\Sbf_\Fin$-dominating, i.e.\ any $\bar c\in\omegabar$ is $\subseteq^\Fin$-bounded by some $\bar d\in E$.
\end{proof}

\begin{corollary}\label{addN-covM}
The additivites of $\Ncal_J$ and $\Ncal^*_J$ are between $\add(\Ncal)$ and $\cov(\Mcal)$ (at the bottom of Cicho\'n's diagram); the cofinalities of $\Ncal_J$ and $\Ncal^*_J$ are between $\non(\Mcal)$ and $\cof(\Ncal)$ (at the top of Cicho\'n's diagram).
\end{corollary}

\begin{proof}
Since $\Fin\leqKB J$, by \autoref{KBomegabar} we obtain that $\bfrak_\Fin(\omegabar)\leq \bfrak_J(\omegabar)$ and $\dfrak_J(\omegabar)\leq\dfrak_J(\omegabar)$. Hence, by \autoref{omegabarvsNJ} and~\ref{omegabarFin}, we obtain 
\[\begin{split}
\add(\Ncal)\leq &\ \bfrak_J(\omegabar)\leq \min\{\add(\Ncal_J),\add(\Ncal^*_J)\}  \text{ and }\\ \max\{\cof(\Ncal_J),\cof(\Ncal^*_J)\}\leq &\ \dfrak_J(\omegabar)\leq \cof(\Ncal). 
\end{split}\]

On the other hand, by \autoref{cofIJrel} and \autoref{thm:E}, $\add(\Ncal_J)=\add(\Ncal_J,\Ncal_J)\leq \add(\Ecal,\Ncal)=\cov(\Mcal)$ and $\non(\Mcal)= \cof(\Ecal,\Ncal) \leq \cof(\Ncal_J,\Ncal_J)=\cof(\Ncal_J)$, likewise for $\Ncal^*_J$.
\end{proof}

\section{Consistency results}\label{sec:cons}

We show the behaviour of the cardinal characteristics associated with $\Ncal_J$ and $\Ncal^*_J$ in different forcing models. As usual, we start with the Cohen model, where the behaviour of these cardinal characteristics are similar to $\bfrak_J$ and $\dfrak_J$, in the sense of~\cite{Canjar}. Inspired by this reference, we present the following effect of adding a single Cohen real.

\begin{lemma}\label{cohenlemma}
Cohen forcing $\Cbb$ adds a real $\bar e\in\omegabar$ such that, for any ideal $J$ on $\omega$ in the ground model, there is some ideal $J'\supseteq J$ on $\omega$ in the generic extension, such that $\bar c\subseteq^{J'} \bar e$ for any $\bar c\in\omegabar$ in the ground model.
\end{lemma}
\begin{proof}
Consider Cohen forcing $\Cbb$ as the poset formed by pairs of finite  sequences $p=(n^p,c^p)$  such that $n^p=\la n^p_k:\, k<m \ra$ is an increasing sequence of natural numbers, $c^p= \la c^p_i :\, i<n^p_{m-1}\ra$ is a sequence of clopen subsets of $\cantor$ and 
\begin{equation}\label{eq:cohencond}
    \mu\left(\bigcup_{i=n^p_k}^{n^p_{k+1}-1} c^p_i\right)<2^{-(k+1)}
\end{equation}
for any $k<m-1$. The order is $q\leq p$ iff $n^q$ end-extends $n^p$ and $c^q$ end-extends $c^p$. If $G$ is $\Cbb$-generic over the ground model $V$, we define $\bar e$ by $e_k:=c^p_k$ for some $p\in G$ (this value does not depend on such a $p$). It is easy to show that $\bar e \in \omegabar$. 

It is enough to show that, in the Cohen extension,  \[J\cup\set{\set{k<\omega}{c_k\nsubseteq e_k}}{\bar c\in \omegabar\cap V} \text{ has the finite union property,}\] 
i.e.\ $\omega$ cannot be covered by finitely many members of that collection. Then, the promised $J'$ will be the ideal generated by this collection.

So, in the ground model, fix $a\in J$, $F\subseteq \omegabar$ finite and $p\in\Cbb$ with $n^p=\la n^p_k:\, k<m \ra$ and $c^p= \la c^p_i :\, i<n^p_{m-1}\ra$. We have to show that there is some $q\leq p$ and some $k<\omega$ such that $q$ forces 
\[k\notin a\cup\bigcup_{\bar c\in F}\set{k<\omega}{c_k\nsubseteq e_k},\]
that is, $k\notin a$ and $c_k\subseteq c^q_k$ for any $\bar c\in F$. To see this, find some $n'\in \omega\smallsetminus a$ larger than $n^p_{m-1}$ such that, for any $\bar c\in F$,
\[\mu\left(\bigcup_{k\geq n'}c_k\right)<\frac{1}{(|F|+1)2^{m+1}}.\]
Define $q\in\Cbb$ such that $n^q$ has length $m+2$, it extends $n^p$, $n^q_m:=n'$ and $n^q_{m+1}:=n'+1$, and such that $c^q$ extends $c^p$, $c^q_i=\emptyset$ for all $n^p_{m-1}\leq i< n'$, and $c^q_{n'}:= \bigcup_{\bar c\in F}c_{n'}$. It is clear that $q\in\Cbb$ is stronger than $p$, and that $c^q_{n'}$ contains $c_{n'}$ for all $\bar c\in F$. So $k:=n'$ works.
\end{proof}

Since FS (finite support) iterations of (non-trivial) posets adds Cohen reals at limit steps, we have the following general consequence of the previous lemma.

\begin{theorem}\label{thm:FSit}
Let $\pi$ be a limit ordinal with uncountable cofinality and let $\Pbb=\la \Pbb_\alpha,\Qnm_\alpha:\, \alpha<\pi\ra$ be a FS iteration of non-trivial $\cf(\pi)$-cc posets. Then, $\Pbb$ forces that there is some (maximal) ideal $J$ such that $\bfrak_{J}(\omegabar)=\add(\Ncal_{J})=\add(\Ncal^*_{J})=\cof(\Ncal^*_{J})= \cof(\Ncal_{J}) =\dfrak_{J}(\omegabar)=\cf(\pi)$.
\end{theorem}
\begin{proof}
Let $L:=\{0\}\cup\set{\alpha<\pi}{\alpha \text{ limit}}$.
For each $\alpha\in L$ let $\bar e_\alpha$ be a $\Pbb_{\alpha+\omega}$-name of a Cohen real in $\omegabar$ (in the sense of the proof of \autoref{cohenlemma}) over the $\Pbb_\alpha$-extension. We construct, by recursion, a sequence $\la \dot J_\alpha:\, \alpha\in L\cup\{\pi\}\ra$ such that $\dot J_\alpha$ is a $\Pbb_\alpha$-name of an ideal on $\omega$ and $\Pbb$ forces that $\dot J_\alpha\subseteq \dot J_\beta$ when $\alpha<\beta$. We let $\dot J_0$ be (the $\Pbb_0$-name of) any ideal $J_0$ in the ground model.\footnote{Recall that $\Pbb_0$ is the trivial poset, so its generic extension is the ground model itself.} For the successor step, assume we have constructed $\dot J_\alpha$ at a stage $\alpha\in L$. By \autoref{cohenlemma}, we obtain a $\Pbb_{\alpha+\omega}$-name $\dot J_{\alpha+\omega}$ of an ideal extending $\dot J_{\alpha}$, such that $\bar e_\alpha$ $\subseteq^{\dot J_{\alpha+\omega}}$-dominates all the $\bar c\in\omegabar$ from the $\Pbb_{\alpha}$-extension. 
For the limit step, when $\gamma$ is a limit point of $L$ (which includes the case $\gamma=\pi$), just let $\dot J_\gamma$ be the $\Pbb_\gamma$-name of $\bigcup_{\alpha<\gamma}\dot J_\alpha$. This finishes the construction.

We show that $\dot J_\pi$ is as required.
In the final generic extension, since $\bar e_\alpha \subseteq^{J_{\beta+\omega}} \bar e_\beta$ for all $\alpha<\beta$ in $L$ and $J_{\beta+\omega}\subseteq J_\pi$, we obtain $\bar e_\alpha \subseteq^{J_\pi} \bar e_\beta$. To conclude $\bfrak_{J_\pi}(\omegabar)=\dfrak_{J_\pi}(\omegabar)=\cf(\pi)$, it remains to show that $\set{\bar e_\alpha}{\alpha\in L}$ is $\subseteq^{J_\pi}$-dominating (because $L$ is cofinal in $\pi$, and the rest follows by \autoref{omegabarvsNJ}). Indeed, if $\bar c\in\omegabar$ in the final extension, then $\bar c$ is in some intermediate extension at $\alpha\in L$, so $\bar c \subseteq^{J_{\alpha+\omega}} \bar e_\alpha$ and hence $\bar c \subseteq^{J_\pi} \bar e_\alpha$.

It is then clear that $\bfrak_{J_\pi}(\omegabar)\leq\dfrak_{J_\pi}(\omegabar)\leq\cf(\pi)$. For $\cf(\pi)\leq\bfrak_{J_\pi}(\omegabar)$, if $F\subseteq \omegabar$ has size ${<}\cf(\pi)$, then we can find, for each $\bar c\in F$, some $\alpha_{\bar c}\in L$ such that $\bar c\subseteq^{J_\pi} \bar e_{\alpha_{\bar c}}$. Since $|F|<\cf(\pi)$, there is some $\beta\in L$ larger than $\alpha_{\bar c}$ for all $\bar c\in F$, so $\bar e_{\alpha_{\bar c}}\subseteq^{J_\pi} \bar e_\beta$. Then $\bar e_\beta$ is a $\subseteq^{J_\pi}$-upper bound of $F$ (because $\subseteq^{J_\pi}$ is a transitive relation).

Note that any (maximal) ideal extending $J_\pi$ also satisfies the conclusion (by \autoref{KBomegabar}).
\end{proof}

The previous results gives a lot of information about the effect of adding many Cohen reals.

\begin{theorem}\label{cohenthm}
Let $\lambda$ be an uncountable cardinal. Then $\Cbb_\lambda$ forces that, for any regular $\aleph_1\leq\kappa\leq\lambda$, there is some (maximal) ideal $J^\kappa$ on $\omega$ such that $\bfrak_{J^\kappa}(\omegabar)=\add(\Ncal_{J^\kappa})=\add(\Ncal^*_{J^\kappa})=\cof(\Ncal^*_{J^\kappa})= \cof(\Ncal_{J^\kappa}) =\dfrak_{J^\kappa}(\omegabar)=\kappa$.
\end{theorem}
\begin{proof}
Let $\kappa$ be a regular cardinal between $\aleph_1$ and $\lambda$. Recall that $\Cbb_\lambda$ is forcing equivalent with $\Cbb_{\lambda+\kappa}$, so we show that the latter adds the required $J^\kappa$. In fact, $\Cbb_{\lambda+\kappa}$ can be seen as the FS iteration of $\Cbb$ of length $\lambda+\kappa$. Since $\cf(\lambda+\kappa)=\kappa$, by \autoref{thm:FSit} we get that $\Cbb_{\lambda+\kappa}$ adds the required $J^\kappa$. 
Note that any (maximal) ideal extending $J^\kappa$ also satisfies the conclusion (by \autoref{KBomegabar}).
\end{proof}

Using sums of ideals, we can obtain from the previous theorem that, after adding many Cohen reals, there are non-meager ideals $K$ satisfying $\non(\Ncal^*_K)<\non(\Ncal_K)$ and $\cov(\Ncal_K)<\cov(\Ncal^*_K)$ (\autoref{cor:CohenI+J}). Before proving this, we calculate the cardinal characteristics associated with some operations of ideals.

\begin{lemma}\label{lem:intid}
  Let $\Icl$ and $\Jcl$ be ideals on an infinite set $X$. Then:
  \begin{enumerate}[label= \rm (\alph*)]
      \item\label{it:addIcapJ} $\min\{\add(\Icl),\add(\Jcl)\} \leq \add(\Icl\cap\Jcl)$ and $\cof(\Icl\cap\Jcl) \leq \max\{\cof(\Icl),\cof(\Jcl)\}$.
      \item $\non(\Icl\cap\Jcl)=\min\{\non(\Icl), \non(\Jcl)\}$ and $\cov(\Icl\cap\Jcl) = \max\{\cov(\Icl),\cov(\Jcl)\}$.
  \end{enumerate}
  For the following items, assume that $\Icl\cup\Jcl$ generates an ideal $\Kcl$.
  \begin{enumerate}[resume*]
      \item\label{it:gencof} $\min\{\add(\Icl),\add(\Jcl)\} \leq \add(\Kcl)$ and $\cof(\Kcl) \leq \max\{\cof(\Icl),\cof(\Jcl)\}$.
      \item\label{it:gencov} $\max\{\non(\Icl),\non(\Jcl)\} \leq \non(\Kcl)$ and $\cov(\Kcl) \leq \min\{\cov(\Icl),\cov(\Jcl)\}$.
  \end{enumerate}
\end{lemma}
\begin{proof}
We use the product of relational systems to shorten this proof. If $\Rbf=\la X,Y,R\ra$ and $\Rbf'=\la X',Y',R'\ra$ are relational systems, we define $\Rbf\times\Rbf':=\la X\times X', Y\times Y', R^\times\ra$ with the relation $(x,x')\, R^\times\, (y,y')$ iff $x\, R\, y$ and $x'\, R'\, y'$. Recall that $\bfrak(\Rbf\times\Rbf')=\min\{\bfrak(\Rbf),\bfrak(\Rbf')\}$ and $\max\{\dfrak(\Rbf),\dfrak(\Rbf')\}\leq \dfrak(\Rbf\times\Rbf')\leq \dfrak(\Rbf)\cdot \dfrak(\Rbf')$ (so equality holds when some $\dfrak$-number is infinite), see e.g.~\cite[Sec.~4]{blassbook}.

As relational systems, it is easy to show that $\Icl\cap\Jcl \leqT \Icl\times\Jcl$ and $\Cbf_{\Icl\cap\Jcl}\leqT \Cbf_\Icl\times \Cbf_\Jcl$, which implies~\ref{it:addIcapJ}, $\min\{\non(\Icl), \non(\Jcl)\}\leq \non(\Icl\cap\Jcl)$ and $\cov(\Icl\cap\Jcl) \leq \max\{\cov(\Icl),\cov(\Jcl)\}$. The converse inequality for the uniformity and the covering follows by \autoref{corcovsub}, as well as~\ref{it:gencov}.

We can also show that $\Kcl\leqT \Icl\times\Jcl$, 
which implies~\ref{it:gencof}. 
\end{proof}

\begin{corollary}[of \autoref{cohenthm}]\label{cor:CohenI+J}
The poset $\Cbb_\lambda$ forces that, for any regular $\aleph_1\leq\kappa_1\leq\kappa_2\leq\lambda$, there is some non-meager ideal $K$ 
such that 
\begin{linenomath}
\begin{align*}
\add(\Ncal^*_K)= \add(\Ncal_K) = \non(\Ncal^*_K) = \cov(\Ncal_K) & = \kappa_1,\\
\cof(\Ncal^*_K)= \cof(\Ncal_K) = \cov(\Ncal^*_K) = \non(\Ncal_K) & = \kappa_2.
\end{align*}
\end{linenomath}
\end{corollary}
\begin{proof}
In the $\Cbb_\lambda$-generic extension, let $J^\kappa$ be a maximal ideal as in \autoref{cohenthm}. We show that $K:= J^{\kappa_1}\oplus J^{\kappa_2}$ is the required ideal. By \autoref{exm:sum}, $\Ncal^*_K=\Ncal^*_{J^{\kappa_1}}\cap \Ncal^*_{J^{\kappa_2}}$ and $\Ncal_K$ is the ideal generated by $\Ncal^*_{J^{\kappa_1}}\cup \Ncal^*_{J^{\kappa_2}}$, so by \autoref{lem:intid} we can perform the following calculations:
\[\begin{split}
    \kappa_1 & =\min\{\add(\Ncal^*_{J^{\kappa_1}}),\add(\Ncal^*_{J^{\kappa_2}})\} \leq \add(\Ncal^*_K)\\
    & \leq \non(\Ncal^*_K)=\min\{\non(\Ncal^*_{J^{\kappa_1}}),\non(\Ncal^*_{J^{\kappa_2}})\}=\kappa_1;\\[1ex]
    \kappa_1 & =\min\{\add(\Ncal_{J^{\kappa_1}}),\add(\Ncal_{J^{\kappa_2}})\} \leq \add(\Ncal_K)\\
    &\leq \cov(\Ncal_K) \leq \min\{\cov(\Ncal_{J^{\kappa_1}}),\cov(\Ncal_{J^{\kappa_2}})\}=\kappa_1;\\[1ex]
    \kappa_2 & =\max\{\cov(\Ncal^*_{J^{\kappa_1}}),\cov(\Ncal^*_{J^{\kappa_2}})\} = \cov(\Ncal^*_K)\\
    & \leq \cof(\Ncal^*_K) \leq \max\{\cof(\Ncal^*_{J^{\kappa_1}}),\cof(\Ncal^*_{J^{\kappa_2}})\}=\kappa_2;\\[1ex]
    \kappa_2 & =\max\{\non(\Ncal_{J^{\kappa_1}}),\non(\Ncal_{J^{\kappa_2}})\} \leq \non(\Ncal_K)\\
    & \leq \cof(\Ncal_K) \leq \max\{\cof(\Ncal_{J^{\kappa_1}}),\cof(\Ncal_{J^{\kappa_2}})\}=\kappa_2.\qedhere
\end{split}\]
\end{proof}

We do not know how to force that there is some non-meager ideal $K$ such that $\add(\Ncal^*_K)\neq \add(\Ncal_K)$, likewise for the cofinality.

Using \autoref{cohenthm} and well-known forcing models, we can show that ZFC cannot prove more inequalities of the cardinal characteristics associated with our new ideals with the classical cardinal characteristics of the continuum of \autoref{fig:all20}, 
but leaving some few open questions. We skip most of the details in the following items, but the reader can refer to the definition of the cardinal characteristics and their inequalities in~\cite{BJ,blassbook,BHH}, and learn the forcing techniques from e.g.
~\cite{Br,BJ,blassbook,Mmatrix,mejia2}.

\begin{figure}
\begin{tikzpicture}
\small{
 \node (aleph1) at (-1.5,-1) {$\aleph_1$};
 \node (addn) at (0,0){$\add(\Ncal)$};
 \node (adNJ) at (1,2) {$\add(\Ncal_J)$};
 \node (adNJ*) at (3,3) {$\add(\Ncal^*_J)$};
 \node (covn) at (0,10){$\cov(\Ncal)$};
 \node (cvNJ) at (1,8) {$\cov(\Ncal_J)$};
 \node (cvNJ*) at (3,7) {$\cov(\Ncal^*_J)$};
 \node (cove) at (5,6) {$\cov(\Ecal)$};
 \node (b) at (4,5) {$\bfrak$};
 \node (addm) at (4,0) {$\add(\Mcal)$} ;
 \node (nonm) at (4,10) {$\non(\Mcal)$} ;
 \node (none) at (7,4) {$\non(\Ecal)$};
 \node (d) at (8,5) {$\dfrak$};
 \node (covm) at (8,0) {$\cov(\Mcal)$} ;
 \node (cfm) at (8,10) {$\cof(\Mcal)$} ;
 \node (nonn) at (12,0) {$\non(\Ncal)$} ;
 \node (nnNJ) at (11,2) {$\non(\Ncal_J)$};
 \node (nnNJ*) at (9,3) {$\non(\Ncal^*_J)$};
 \node (cfn) at (12,10) {$\cof(\Ncal)$} ;
 \node (cfNJ) at (11,8) {$\cof(\Ncal_J)$};
 \node (cfNJ*) at  (9,7){$\cof(\Ncal^*_J)$};
 \node (s) at (-1,3) {$\sfrak$};
 \node (r) at (13,7) {$\rfrak$};
 \node (c) at (13.5,10) {$\cfrak$};
 \node (m) at (-1.5,0) {$\mathfrak m$};
 \node (p) at (-1.25,1) {$\mathfrak p$};
 \node (h) at (-1.125,2) {$\mathfrak h$};
 \node (g) at (-1.5,3.5) {$\mathfrak g$};
 \node (e) at (3,1) {$\mathfrak e$};
 \node (a) at (13.5,6.5) {$\mathfrak a$};
 \node (u) at (13.25,8.5) {$\mathfrak u$};
 \node (i) at (12.75,9) {$\mathfrak i$};

\foreach \from/\to in {
m/addn, cfn/c, aleph1/m, m/p, p/h, h/s, h/g, r/u, u/c,
h/b, p/addm, g/d, p/e, addn/e, e/none, e/covm,
b/a, a/c, cfm/i, r/i, i/c,
s/none, s/d, b/r, cove/r,
adNJ/nnNJ, adNJ*/nnNJ*, cvNJ/cfNJ, cvNJ*/cfNJ*,
addm/b, b/nonm, b/d, addm/covm, covm/d, nonm/cfm, d/cfm,
addm/none, cove/cfm, 
covm/cove, none/nonm,
addn/adNJ, addn/adNJ*,
covn/cvNJ, cvNJ/cvNJ*, cvNJ*/cove,
addn/covn, adNJ/cvNJ, adNJ*/cvNJ*,
addn/addm, covn/nonm, covm/nonn, cfm/cfn, nonn/cfn,
adNJ/covm, adNJ*/covm,
none/nnNJ*, nnNJ*/nnNJ, nnNJ/nonn,
nnNJ/cfNJ, nnNJ*/cfNJ*,
cfNJ*/cfn, cfNJ/cfn,
nonm/cfNJ*, nonm/cfNJ}
{
\path[-,draw=white,line width=3pt] (\from) edge (\to);
\path[->,] (\from) edge (\to);
}
}
\end{tikzpicture}
\caption{Cicho\'n's diagram and the Blass  diagram combined, also including the coverings and uniformities of our new ideals. An arrow $\mathfrak x\rightarrow \mathfrak y$ means that ZFC proves $\mathfrak x\le \mathfrak y$.}\label{fig:all20}
\end{figure}
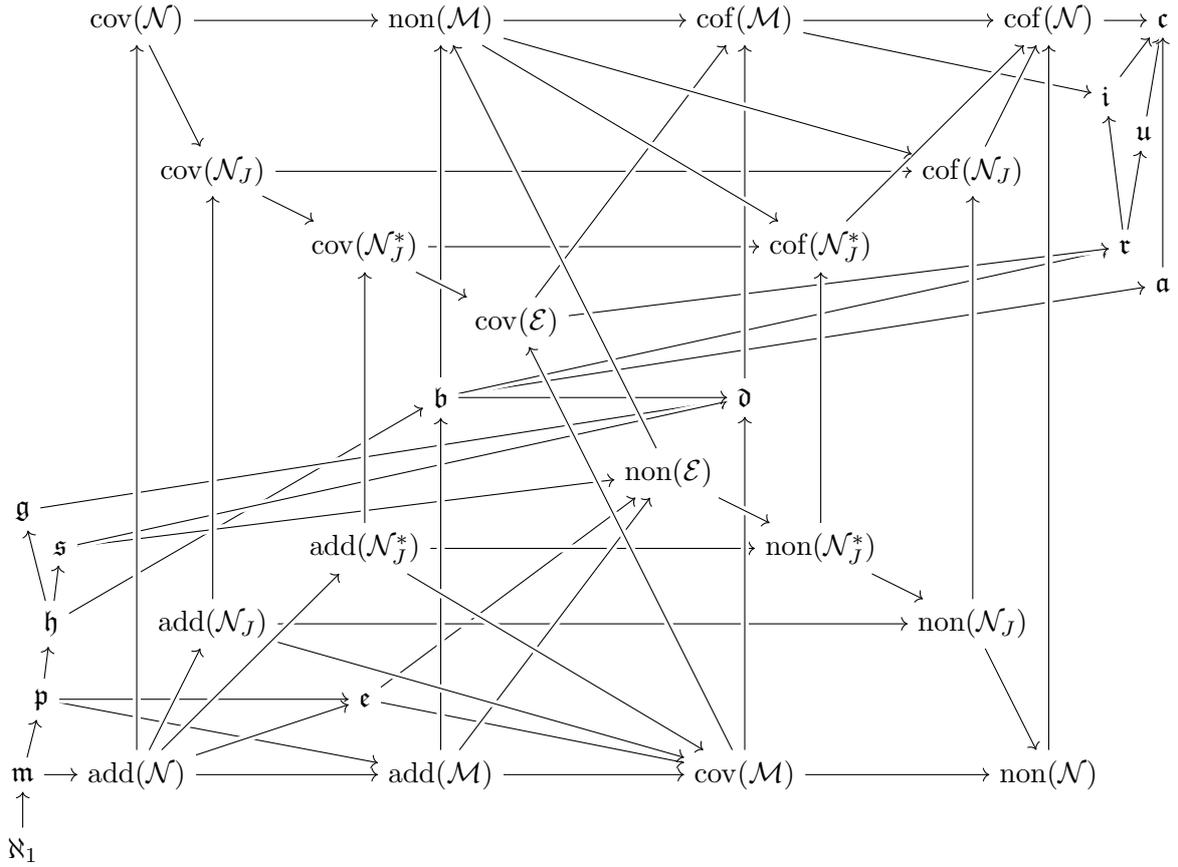

\begin{enumerate}[label=\rm (M\arabic*)]
    \item\label{modelCohen} Using $\lambda>\aleph_2$, $\Cbb_\lambda$ forces that there is some (maximal) ideal $J$ on $\omega$ such that $\non(\Mcal)=\gfrak =\afrak=\aleph_1< \add(\Ncal_J)=\cof(\Ncal_J) =\add(\Ncal^*_J) =\cof(\Ncal^*_J)<\cov(\Mcal)$ (see \autoref{cohenthm}).
    
    \item We can iterate the Hechler poset, followed by a large random algebra, to force $\non(\Ncal)=\aleph_1<\bfrak=\dfrak=\afrak<\cov(\Ncal)=\cfrak$. In this generic extension, we obtain $\add(\Ncal_J)= \add(\Ncal^*_J) = \non(\Ncal_J)= \non(\Ncal^*_J)=\gfrak =\aleph_1<\bfrak=\dfrak=\afrak<\cov(\Ncal_J)= \cov(\Ncal^*_J) = \cof(\Ncal_J)= \cof(\Ncal^*_J)=\cfrak$ for any ideal $J$ on $\omega$. This idea to force with a random algebra after some other FS iteration of ccc posets is original from Brendle, but some details can be found in~\cite[Sec.~5]{collapsing}.
    
    \item\label{M:Miller} In the Miller model, $\non(\Mcal)=\non(\Ncal)=\ufrak=\afrak=\aleph_1<\gfrak=\cfrak=\aleph_2$, so \ref{NCF} follows. Hence $\add(\Ncal_J)= \add(\Ncal^*_J) = \non(\Ncal_J)= \non(\Ncal^*_J)= \cov(\Ncal_J) = \cov(\Ncal^*_J)=\aleph_1 < \gfrak=\cfrak=\aleph_2$ for any ideal $J$ on $\omega$, e.g., since ZFC proves $\cov(\Ecal)\leq \rfrak\leq \ufrak$, $\cov(\Ecal)=\aleph_1$ in the Miller model. Although $\cof(\Ncal_J)$ and $\cof(\Ncal^*_J)$ are $\aleph_2$ when $J$ has the Baire property, we do not know what happens when $J$ does not have the Baire property (or just in the case of maximal ideals).
    
    \item In the Mathias model, $\cov(\Ecal)=\aleph_1<\hfrak=\cfrak=\aleph_2$ (the first equality due to the Laver property). Hence $\add(\Ncal_J)= \add(\Ncal^*_J)= \cov(\Ncal_J) = \cov(\Ncal^*_J)=\aleph_1 < \non(\Ncal_J)= \non(\Ncal^*_J)= \cof(\Ncal_J) = \cof(\Ncal^*_J)= \aleph_2$ (notice that $\hfrak\leq \sfrak\leq \non(\Ecal)$).
    
    \item\label{C:Hechler} Assume $\kappa$ and $\lambda$ cardinals such that $\aleph_1\leq \kappa\leq\lambda=\lambda^{\aleph_0}$. The FS iteration of Hechler forcing of length $\lambda\kappa$ (ordinal product) forces $\efrak= \gfrak=\sfrak =\aleph_1$, $\add(\Mcal)=\cof(\Mcal)=\kappa$ and $\non(\Ncal)=\rfrak=\cfrak=\lambda$. Thanks to \autoref{thm:FSit}, 
    there is a (maximal) ideal $J^\kappa$ on $\omega$ such that $\add(\Ncal_{J^\kappa})=\add(\Ncal^*_{J^\kappa})=\cof(\Ncal^*_{J^\kappa})= \cof(\Ncal_{J^\kappa})=\kappa$. In general, we can just say that the additivities and coverings of $\Ncal_J$ and $\Ncal^*_J$ are below $\kappa$, and that their uniformities and cofinalities are above $\kappa$ for any ideal $J$ on $\omega$.
    
    Since the Hechler poset makes the set of reals from the ground model meager, we have that any ideal $J_0$ on $\omega$ with a set of generators in the ground model (or in any intermediate extension) is meager in the final extension, so $\cov(\Ncal_{J_0})=\aleph_1$ and $\non(\Ncal_{J_0})=\cfrak=\lambda$. We still do not know how to obtain a maximal (or just non-meager) ideal $J'$ in the final extension such that $\cov(\Ncal_{J'})=\aleph_1$ and $\non(\Ncal_{J'})=\lambda$. 
    
    \item\label{M:sigmac} With $\kappa$ and $\lambda$ as in~\ref{C:Hechler}, there is a FS iteration of length $\lambda\kappa$ of $\sigma$-centered posets forcing $\cov(\Ncal)=\aleph_1$, $\pfrak=\ufrak=\afrak=\ifrak=\kappa$ and $\non(\Ncal)=\cfrak=\lambda$. For example, by counting arguments, we make sure to force with all $\sigma$-centered posets of size ${<}\kappa$, while adding witnesses of $\ufrak$, $\afrak$ and $\ifrak$ of size $\kappa$ using a cofinal subset of $\lambda\kappa$ of size $\kappa$ (for $\ufrak$, use Mathias-Prikry forcing with ultrafilters, and for $\afrak$ and $\ifrak$ use Hechler-type posets as in \cite{Hechlermad}, all of which are $\sigma$-centered).    
    We have exactly the same situation of the previous model for the cardinal characteristics associated with $\Ncal_J$ and $\Ncal^*_J$.
    
    \item Over a model of CH, the countable support iteration of length $\omega_2$ of the tree forcing from~\cite[Lem.~2]{brendle99} forces $\dfrak= \cov(\Ecal)=\aleph_1<\non(\Ecal)=\aleph_2$. Concretely, this forcing is proper, $\baire$-bounding, and does not add random reals (see also~\cite[Sec.~7.3B]{BJ}), hence $\cov(\Ecal)\leq\max\{\dfrak,\cov(\Ncal)\}=\aleph_1$ (see \autoref{thm:E}). In this generic extension, all additivites and coverings of $\Ncal_J$ and $\Ncal^*_J$ are $\aleph_1$, while the uniformities and cofinalities are $\aleph_2$.
    
    \item Brendle's model~\cite{brendle02} of $\cof(\Ncal)<\afrak$ clearly satisfies that the cardinal characteristics associated with $\Ncal_J$ and $\Ncal^*_J$ are strictly smaller than $\afrak$ for every ideal $J$ on $\omega$.
\end{enumerate}

It is not clear from the above whether $\gfrak \leq \cof(\Ncal_J)$ in ZFC for all $J$, and whether we can force $\cov(\Ncal_J)<\pfrak$ and $\non(\Ncal_J)>\max\{\ufrak,\ifrak\}$ for some non-meager $J$.

\section{Some variations}\label{sec:var}

\subsection{Measure zero modulo ideals with respect to other Polish spaces}\label{subsec:Polish}

Our notion of $\Ncal_J$ and $\Ncal^*_J$ has been studied for the Cantor space $\cantor$, but how about other Polish spaces (with a measure)? We look at the same notions for other classical spaces like the real line $\R$, the unit interval $[0,1]$, and product spaces of the form $\prod b$ for some sequence $b=\la b(n):\, n<\omega\ra$ of finite discrete spaces such that $\set{n<\omega}{|b(n)|>1}$ is infinite.

We start with $\R$. Let $\Omega_{\R}$ be the collection of open subsets of $\R$ that can be expressed as a finite union of open intervals and, for any ideal $J$ on $\omega$ and $\bar c\in {}^\omega \pts(\R)$, define
\[\begin{split}
N^\R_J(\bar c) := \set{x\in \R}{ \set{n<\omega}{x\in c_n}\notin J},\\
N^{*\R}_J(\bar c) := \set{x\in \R}{ \set{n<\omega}{x\notin c_n}\in J}.
\end{split}\]
The upper index $\R$ can be omitted when we know we are working in $\R$.

We can characterize, as in \autoref{chomegabar}, the $\bar c\in {}^\omega \Omega_{\R}$ such that $N_\Fin(\bar c)$ has Lebesgue measure zero. Denote the collection of such $\bar c$ by $\omegabar_\R$
and define $\Ncal_J(\R)$ and $\Ncal^*_J(\R)$ in the natural way. As in \autoref{basicNE}, we obtain that $\Ncal_\Fin(\R)$ is $\Ncal(\R)$, the collection of measure zero subsets of $\R$. 

It is possible to restrict $\omegabar_\R$ to any (countable) base $B$ of open intervals, without affecting the resulting $\Ncal_J(\R)$ and $\Ncal^*_J(\R)$. Let $\omegabar_B$ be the collection of $\bar d\in\omegabar_\R$ such that each $d_n$ is a finite union of members of $B$. If $\bar c\in\omegabar_\R$, then only finitely many $c_n$ are unbounded, for which we define $d_n:=\emptyset$; for each bounded $c_n$, we can find some $d_n$ such that $c_n\subseteq d_n$, $d_n\smallsetminus c_n$ has measure ${<}2^{-n}$, and $d_n$ is a finite union of members of $B$. Then $\bar d\in\omegabar_B$, and it is clear that $N_J(\bar c)\subseteq N_J(\bar d)$ and $N^*_J(\bar c)\subseteq N^*_J(\bar d)$.

Also, as in \autoref{charE},

\begin{lemma}\label{charE:R}
  The collection $\Ncal^*_\Fin(\R)$ coincides with $\Ecal(\R)$, the ideal generated by the $F_\sigma$ measure zero subsets of $\R$.
\end{lemma}
\begin{proof}
If $\bar c\in\omegabar_\R$ then $\bar c':=\la \cl(c_n):\, n<\omega\ra$ (where $\cl(A)$ denotes the topological closure of $A$) still satisfies that the measure of $\bigcup_{n\geq m}\cl(c_n)$ tends to $0$ as $m\to\infty$, so $N^{*\R}_\Fin(\bar c)\subseteq N^{*\R}_\Fin(\bar c')$ and, clearly, $N^{*\R}_\Fin(\bar c')\in \Ecal(\R)$, so $N^{*\R}_\Fin(\bar c)\in \Ecal(\R)$. This shows $\Ncal^*_\Fin(\R)\subseteq\Ecal(\R)$.

For the converse contention, assume that $X=\bigcup_{n<\omega} F_n$ where each $F_n$ is a closed measure zero subset of $\R$. Without loss of generality, we can assume that the sequence $\la F_n:\, n<\omega\ra$ is increasing and that each $F_n$ is bounded (hence, compact). As in the proof of \autoref{basicE}, if we fix a sequence $\la \varepsilon_n:\, n<\omega\ra$ of positive real numbers such that $\sum_{i<\omega}\varepsilon_i<\infty$, we can cover each $F_n$ with some $c_n\in \Omega_\R$ of measure ${<}\varepsilon_n$. Then $\bar c\in\omegabar_\R$ and $X\subseteq N^*_\Fin(\bar c)$.
\end{proof}

As in \autoref{NJsigma2}, both $\Ncal_J(\R)$ and $\Ncal^*_J(\R)$ are $\sigma$-ideals on $\R$. Although all the results of this work can be verified for these ideals, these follow by the connection between the ideals and those ideals corresponding to the Cantor space. Before showing this connection, let us look at the case of other Polish spaces.

Consider the closed interval $[a,b]$ with $a<b$ in $\R$. For any base $B$ of open intervals with endpoints in $[a,b]$ (including some intervals of the form $[a,x)$ and $(y,b]$, and even allowing $[a,b]$ itself), we define $\omegabar_B$ as before, as well as $N^{[a,b]}_J(\bar c)$ and $N^{*[a,b]}_J(\bar c)$ (for any $\bar c\in {}^\omega \pts([a,b])$. It is easy to check that
\[
\begin{split}
    \set{X\subseteq [a,b]}{\exists\, \bar c\in \omegabar_B\colon X\subseteq N^{[a,b]}_J(\bar c)} & = \set{Y\cap [a,b]}{Y\in \Ncal_J(\R)},\\
    \set{X\subseteq [a,b]}{\exists\, \bar c\in \omegabar_B\colon X\subseteq N^{*[a,b]}_J(\bar c)} & = \set{Y\cap [a,b]}{Y\in \Ncal^*_J(\R)},
\end{split}
\]
which we denote by $\Ncal_J([a,b])$ and $\Ncal^*_J([a,b])$, respectively. It is clear that these are $\sigma$-ideals.

Now let $b=\la b(n):\, n<\omega\ra$ be a sequence of finite discrete spaces such that $|b(n)|>1$ for infinitely many $n$, and consider the perfect compact Polish space $\prod b$ (notice that the Cantor space is a particular case). The clopen sets $[s]_b:=\set{x\in\prod b}{s\subseteq x}$ for $s\in\sqw b:=\bigcup_{n<\omega}\prod_{i<n}b(i)$ form a base of $\prod b$, and every clopen subset of $\prod b$ is a finite union of basic clopen sets. The Lebesgue measure $\mu_b$ on $\prod b$ is the unique measure on (the completion of) the Borel $\sigma$-algebra such that $\mu_b([s]_b)=\prod_{i<|n|}\frac{1}{|b(i)|}$ for each $s\in\sqw b$.

Let $\Omega_b$ be the collection of clopen subsets of $\prod b$, and define $N^b_J(\bar c)$ and $N^{*b}_J(\bar c)$ as usual for $\bar c\in {}^\omega \pts(\prod b)$. Let $\omegabar_b$ be the collection of $\bar c\in {}^\omega \Omega_b$ such that $\mu_b(N_\Fin^b(\bar c))=0$, and define $\Ncal_J(\prod b)$ and $\Ncal^*_J(\prod b)$ in the natural way. As in the case of the Cantor space, we can verify that these are $\sigma$-ideals, $\Ncal_\Fin(\prod b)$ is the ideal of measure zero subsets of $\prod b$, and $\Ncal^*_\Fin(\prod b)$ is the ideal generated by the $F_\sigma$ measure zero subsets of $\prod b$.

We show the connection between $\prod b$ and $[0,1]$.
Without loss of generality, assume that each $b(n)$ is a natural number. Consider the map $F_b\colon \prod b\to[0,1]$ defined by
\[F_b(x):=\sum_{i<\omega}\frac{x_i}{b(0)\, b(1)\, \cdots\, b(i)}.\]
This is a continuous onto map, and actually an homeomorphism modulo some countable sets. Concretely, let $\Q_b$ be the set of $\sum_{i<|s|}\frac{x_i}{b(0)\, b(1)\, \cdots\, b(i)}$ for $s\in\sqw b$, which is a countable dense subset of $[0,1]$. Each member of $\Q_b\smallsetminus\{0,1\}$ has two pre-images under $F_b$, while each member of $\{0,1\}\cup((0,1)\smallsetminus \Q_b)$ has exactly one pre-image. Letting $\Q'_b:=F_b^{-1}[\Q_b]$, we obtain that $F_b\frestr(\prod b\smallsetminus \Q'_b)$ is a homeomorphism from $\prod b\smallsetminus \Q'_b$ onto $[0,1]\smallsetminus \Q_b$. Note that $\Q'_b$ is the collection of sequences $x\in \prod b$ such that, for some $m<\omega$, either $(\forall\, i\geq m)\ x_i=0$ or $(\forall\, i\geq m)\  x_i=b(i)-1$.

The $F_b$ gives a connection between the ideals $\Ncal_J$ and $\Ncal^*_J$ for $\prod b$ and $[0,1]$. Using the base on $[0,1]$ of open intervals with extremes in $\Q_b$, we can show that, for any $X\subseteq \prod b$ and $Y\subseteq [0,1]$,
\[
\begin{split}
    X \in \Ncal_J\left(\prod b\right) & \text{ iff } F_b[X] \in \Ncal_J([0,1]),\\
    Y \in \Ncal_J([0,1]) & \text{ iff } F_b^{-1}[Y] \in \Ncal_J\left(\prod b\right),
\end{split}
\]
and likewise for $\Ncal^*_J$. Since the Cantor space is a particular case of $\prod b$, the map $F_{b_2}$ for $b_2(n)=2$ can be used to transfer all the results of this work from the Cantor space to $[0,1]$, and $F_b$ can be used to transfer the results from $[0,1]$ to $\prod b$. For example, the map $F_b$ allows to get the Tukey equivalences $\Ncal_J(\prod b)\eqT \Ncal_J([0,1])$ and $\Cbf_{\Ncal_J(\prod b)} \eqT \Cbf_{\Ncal_J([0,1])}$, and likewise for $\Ncal^*_J$, so the cardinal characteristics associated with $\Ncal_J$ and $\Ncal^*_J$ do not depend on the Polish space we are working on.

Finally, we get a connection between $\R$ and $[0,1]$ as follows. Consider the maps\footnote{We allow $\mathrm{id}_A\colon A\to B$ for any $B\supseteq A$ according to our needs, e.g.\ we use $\mathrm{id}_{[0,1]}\colon [0,1]\to \R$.}

\begin{center}
\begin{tabular}{ll}
   $\mathrm{id}_{A}$ & defined as the identity map on a set $A$,\\[1ex]
   $\pi_{[0,1]}\colon \R\to[0,1]$  &  defined by $\pi_{[0,1]}(x):=x-\lfloor x\rfloor$,\\[1ex]
   $\hat{\pi}_{[0,1]}\colon \pts(\R)\to \pts([0,1])$  &  defined by $\hat{\pi}_{[0,1]}(X):= \pi_{[0,1]}[X]\cup\{1\}$, and\\[1ex]
   $F^\R\colon \pts([0,1])\to \pts(\R)$ &  defined by $F^\R(X):=\bigcup_{n\in \Z}(n+X)$. 
\end{tabular}
\end{center}

These maps can be used to show that $\Ncal_J(\R)\eqT \Ncal_J([0,1])$ and $\Cbf_{\Ncal_J(\R)}\eqT \Cbf_{\Ncal_J([0,1])}$, in detail: note that $\hat{\pi}_{[0,1]}$ sends sets in $\Ncal_J(\R)$ to $\Ncal_J([0,1])$, and $F^\R$ sends sets in $\Ncal_J([0,1])$ to $\Ncal_J(\R)$, therefore,

\begin{center}
\begin{tabular}{ll}
    $(\hat{\pi}_{[0,1]}\frestr\Ncal_J(\R),F^\R\frestr \Ncal_J(\R))$  &  is a Tukey connection for $\Ncal_J(\R)\leqT \Ncal_J([0,1])$;\\[1ex]
    $(\pi_{[0,1]},F^\R\frestr \Ncal_J(\R))$ & corresponds to $\Cbf_{\Ncal_J(\R)}\leqT \Cbf_{\Ncal_J([0,1])}$;\\[1ex]
    $(\mathrm{id}_{\Ncal_J([0,1])},\hat{\pi}_{[0,1]}\frestr\Ncal_J(\R))$ &  corresponds to $\Ncal_J([0,1])\leqT \Ncal_J(\R)$; and\\[1ex]
    $(\mathrm{id}_{[0,1]},\hat{\pi}_{[0,1]}\frestr\Ncal_J(\R))$ & corresponds to $\Cbf_{\Ncal_J([0,1])}\leqT \Cbf_{\Ncal_J(\R)}$.
\end{tabular}
\end{center}

\subsection{Alternative definition}

We explore one alternative weaker definition of $\Ncal_J$ and $\Ncal^*_J$. This is developed in detail by the first author in~\cite{Vthesis}, which we summarize here and compare with the notions defined in this paper. We work in the Cantor space $\cantor$ (but results can be transferred to other Polish spaces by using the connections presented in the previous subsection). Consider
\[\lone{\Omega}:=\set{\bar c\in {}^\omega \Omega}{\sum_{n<\omega}\mu(c_n) < \infty}.\]
We define our ideals with respect to $\lone{\Omega}$, that is,
\begin{equation}
\begin{split}
    \Ncal^-_J & :=\set{X\subseteq\cantor}{\exists\bar c\in \lone{\Omega}\colon X\subseteq N_J(\bar c)}\\
    \Ncal^{-*}_J & :=\set{X\subseteq \cantor}{\exists\bar c\in \lone{\Omega}\colon X\subseteq N^*_J(\bar c)}.
\end{split}
\end{equation}




It is clear that $\Ncal^-_J\subseteq \Ncal_J$ and $\Ncal^{-*}_J\subseteq \Ncal^*_J$, but we do not know whether they can be equal for non-meager ideals. 
Most of the proofs of this work can be easily modified to obtain similar results for these collections.
By \autoref{basicN2} and \autoref{basicE}, we have that $\Ncal^-_\Fin=\Ncal$ and $\Ncal^{-*}_\Fin = \Ecal$, and the proof of \autoref{NJsigma2} can be modified to conclude that $\Ncal^-_J$ and $\Ncal^{-*}_J$ are $\sigma$-ideals.

The behaviour of sum of ideals presented in \autoref{exm:sum} can be repeated in this new context. However, \autoref{RBeq} cannot be exactly repeated because, whenever $\bar c\in \lone{\Omega}$, although the $\bar c'$ defined in the proof is in $\lone{\Omega}$, we cannot guarantee that $\bar c^-$ is in $\lone{\Omega}$ ($\bar c^-$ could repeat some $c_n$'s too many times so that the sum of the measures can diverge). Therefore, we can only state the following:

\begin{theorem}\label{RBeq-}
  Let $\Jcal$ and $\Kcal$ be ideals on $\omega$.
  \begin{enumerate}[label=\rm(\alph*)]
      \item\label{it:underKB-} If $\Kcal\leqKB \Jcal$ then  $\Ncal^-_\Jcal\subseteq\Ncal^-_\Kcal$.
      \item\label{it:underKB'-} If $K\leqKBpr J$ then $\Ncal^{-*}_K\subseteq \Ncal^{-*}_J$.
      \item\label{it:underRB-} If $\Kcal\leqRB \Jcal$ then $\Ncal^{-*}_\Jcal\subseteq\Ncal^{-*}_\Kcal$ and $\Ncal^-_\Jcal\subseteq \Ncal^-_\Kcal$.
  \end{enumerate}
\end{theorem}

Thanks to this theorem, we still have that $J\subseteq K$ implies $\Ncal^{-*}_J\subseteq \Ncal^{-*}_K\subseteq \Ncal^-_K\subseteq \Ncal^-_J$. Moreover:

\begin{theorem}\label{Baire-}
    If $J$ is an ideal on $\omega$ with the Baire property, then $\Ncal^-_J=\Ncal$ and $\Ncal^{-*}_J=\Ecal$.
\end{theorem}
\begin{proof}
    Since $\Fin\leqKB J$, by \autoref{RBeq-}
    it is clear that $\Ncal^-_J\subseteq \Ncal$ and $\Ncal^{-*}_J\subseteq\Ecal$, and $\Ecal=\Ncal_\Fin^{-*}\subseteq \Ncal_J^{-*}$ because $\Fin\subseteq J$, so it remains to show that $\Ncal\subseteq \Ncal^-_J$.
    There is a finite-to-one function $f\in{}^\omega\omega$ such that
    $a\in\fin$ if and only if $f^{-1}[a]\in J$. Define
    \[\varepsilon(n):=\frac{1}{2^n(|f^{-1}(n)|+1)}.\]

Assume $X\in\mathcal{N}$, so by \autoref{basicN2} there is some $\bar{c}\in\Omega^*_\varepsilon$ such that $X\subseteq N(\bar{c})$. Let $c^-_k:= c_{f(k)}$ for $k<\omega$. Then $N(\bar{c})\subseteq N_J(\bar{c}^-)$ because, for any $x\in N(\bar{c})$,
\[
f^{-1}[\set{n\in\omega}{x\in c_n}]=\set{k\in\omega}{x\in c_{f(k)}}\in J^+.
\]
Finally, we have that $\bar{c}^-\in\lone{\Omega}$ because
 \[
 \sum_{k\in\omega}\leb{c^-_k}=\sum_{k\in\omega}\leb{c_{f(k)}}=\sum_{n\in\omega}|f^{-1}(n)|\cdot\leb{c_n}\leq\sum_{n\in\omega}|f^{-1}(n)|\cdot\varepsilon(n)\leq\sum_{n\in\omega}2^{-n}.
 \]
 Clearly, $X\subseteq N_J(\bar c')$, so $X\in\Ncal^-_J$.
%
\end{proof}

However, we lose some results due to the asymmetry in \autoref{RBeq-}. 
The results in \autoref{sec:nBP} are valid in this new context, but we can not repeat everything from \autoref{sec:near} in this new context. We lose part of \autoref{lem:coh}, namely, we cannot guarantee that, whenever $J_0$ and $J_1$ are nearly coherent ideals on $\omega$, there is some ideal $K$ such that $\Ncal^-_K\subseteq \Ncal^-_{J_0}\cap \Ncal^-_{J_1}$ (but we can still say that $\Ncal^{-*}_{J_0}\cup \Ncal^{-*}_{J_1}\subseteq \Ncal^{-*}_K$). As a consequence, \autoref{cor:NCuf} and~\ref{cor:NCFNuf} cannot be guaranteed in this new context, e.g., we cannot claim anymore that \ref{NCF} implies that there is a single $\Ncal^-_J$ for maximal $J$.

The results using the nearly coherence game can be obtained in the new context, in particular, \autoref{notnear2} and~\autoref{cor:nnc-NE} still hold. However, in \autoref{nc-char},~\ref{it:aida} is reduced to $\Ncal^*_J\cup \Ncal^*_K\subseteq \Ncal^*_{K'}$, and it becomes unclear whether this reduced~\ref{it:aida} and~\ref{it:subset} are still equivalent to nearly coherence of $J$ and $K$ (we can just say that~\ref{it:subset} implies nearly coherence, and that nearly coherence implies the reduced~\ref{it:aida}). \autoref{noncoh:sum} is valid.

For the associated cardinal characteristics as discussed in \autoref{sec:cardinalityNJ}, the discussion about the covering and the uniformity is exactly the same, and \autoref{addN-covM} still holds. However, the proof of the latter and the discussion about $\la \omegabar,\subseteq^J\ra$ changes a lot. In this case, we must look at the relational system $\lone{\Omega,J}:=\la \lone{\Omega},\lone{\Omega},\subseteq^J\ra$ and to their associated cardinal characteristics $\bfrak_J(\lone{\Omega}):=\bfrak(\lone{\Omega,J})$ and $\dfrak_J(\lone{\Omega}):=\dfrak(\lone{\Omega,J})$. Although $\lone{\Omega,J}\leqT \Ncal^-_J, \Ncal^{-*}_J$ can be shown as in \autoref{omegabarvsNJ}, we cannot guarantee monotonicity with the orders $\leqKB$, $\leqKBpr$ and $\leqRB$ as in \autoref{KBomegabar}. We only have that $\lone{\Omega,J}\leqT \lone{\Omega,K}$ whenever $K\subseteq J$.

We can still prove that $\bfrak_\Fin(\lone{\Omega})=\add(\Ncal)$ and $\dfrak_\Fin(\lone{\Omega})=\cof(\Ncal)$ by some modification of the proof of \autoref{omegabarFin}, and thus deriving the result in \autoref{addN-covM}, namely, the additivities of $\Ncal^-_J$ and $\Ncal^{-*}_J$ are between $\add(\Ncal)$ and $\cov(\Mcal)$, and that their cofinalities are between $\non(\Mcal)$ and $\cof(\Ncal)$.

The consistency results in \autoref{sec:cons} can be verified for $\Ncal^-_J$, $\Ncal^{-*}_J$, and for $\bfrak_J(\lone{\Omega})$ and $\dfrak_J(\lone{\Omega})$ in place of $\bfrak_J(\omegabar)$ and $\dfrak_J(\omegabar)$, respectively. The point is to change the condition \eqref{eq:cohencond} for the Cohen poset by
\[
\sum_{i=n^p_k}^{n^p_{k+1}-1}\mu(c^p_i)<2^{-k}
\]
to guarantee that the Cohen generic $\bar e$ is in $\lone{\Omega}$.

\section{Open questions}\label{sec:Q}

We address here several open questions of this work.

\begin{question}
   Does some converse of the statements in \autoref{RBeq} hold in ZFC?
\end{question}

Concerning nearly coherence:

\begin{question}\label{Q:NCF}
  Does \ref{NCF} imply that there is only one $\Ncal_J$ for non-meager $J$?
\end{question}

We may also ask whether (under \ref{NCF}) $\Ncal^*_J=\Ncal_J$ for any non-meager $J$. In contrast, we ask:

\begin{question}\label{Q:N*noN}
Can we prove in $\mathrm{ZFC}$ that there is a non-meager ideal $J$ on $\omega$ such that $\Ncal^*_J\neq \Ncal_J$?
\end{question}

Recall that we produced such an example in \autoref{noncoh:sum} under the assumption that there is a pair of non-nearly coherent ideals on $\omega$. The answer to the previous questions would expand our knowledge about the difference between $\Ncal_J$ and $\Ncal_K$ for different $J$ and $K$.

Concerning the cardinal characteristics, we showed that it is possible to have $\non(\Ncal^*_K)<\non(\Ncal_K)$ and $\cov(\Ncal_K)<\cov(\Ncal^*_K)$ for some non-meager ideal $K$ (in Cohen model, see \autoref{cor:CohenI+J}). However, we do not know what happens to the additivities and the cofinalities.

\begin{question}\label{Q:addcof}
   Is it consistent that $\add(\Ncal_J)$ and $\add(\Ncal^*_J)$ are different for some non-meager ideal $J$ on $\omega$? The same is asked for the cofinalities.
\end{question}

\begin{question}\label{Q:addadd*}
   Does $\mathrm{ZFC}$ prove some inequality between $\add(\Ncal_J)$ and $\add(\Ncal^*_J)$? The same is asked for the cofinalities.
\end{question}

The second author~\cite{Mmatrix} has constructed a forcing model where the four cardinal characteristics associated with $\Ncal$ are pairwise different. Cardona~\cite{Cfriendly} has produced a similar model for $\Ecal$. In this context, we ask:

\begin{question}\label{fourNJ}
   Is it consistent with $\mathrm{ZFC}$ that, for some non-meager (or maximal) ideal $J$ on $\omega$, the four cardinal characteristics associated with $\Ncal_J$ are pairwise different?
\end{question}

In relation with the cardinal characteristics in \autoref{fig:all20}, to have a complete answer that no other inequality can be proved for the cardinal characteristics associated with our ideals, it remains to solve:

\begin{question}\label{Q:g}
   Does $\mathrm{ZFC}$ prove $\gfrak\leq\cof(\Ncal_J)$ for any ideal $J$ on $\omega$?
\end{question}

\begin{question}\label{Q:cov}
   Is it consistent that $\cov(\Ncal_J)<\pfrak$ for some maximal ideal $J$?
\end{question}

\begin{question}\label{Q:non}
   Is it consistent that $\max\{\ufrak,\ifrak\}<\non(\Ncal_J)$ for some maximal ideal $J$?
\end{question}

Concerning the structure of our ideals, we ask:

\begin{question}
   What is the intersection of all $\Ncal_J$? What is the union of all $\Ncal^*_J$?
\end{question}

Note that such intersection is $\Ncal_{\min}$ and the union is $\Ncal^*_{\max}$, where
\[\begin{split}
    \Ncal_{\min} & :=\bigcap\set{\Ncal_J}{J \text{ maximal ideal}},\\
    \Ncal^*_{\max} & :=\bigcup\set{\Ncal_J}{J \text{ maximal ideal}}.
\end{split}
\]
According to \autoref{cor:NCFNuf}, under~\ref{NCF}, $\Ncal_{\min}=\Ncal^*_{\max}$. It is curious what would these families be when allowing non-nearly coherent ideals.

We finish this paper with a brief discussion about strong measure zero sets.
Denote by $\SNcal$ the ideal of strong measure zero subsets of $\cantor$. We know that $\SNcal\subseteq \Ncal$ and $\Ecal\nsubseteq \SNcal$ (because perfect subsets of $\cantor$ cannot be in $\SNcal$). Under Borel's conjecture, we have $\SNcal\subseteq\Ecal$, however, $\cov(\SNcal)<\cov(\Ecal)$ holds in Cohen's model~\cite{P90} (see also~\cite[Sec.~8.4A]{BJ}), which implies $\SNcal\nsubseteq \Ecal$. This motivates to ask:

\begin{question}
   Does ZFC prove that there is some non-meager ideal $J$ such that $\SNcal\subseteq \Ncal_J$, or even $\SNcal\subseteq \Ncal^*_J$?
\end{question}

In the model from~\autoref{cohenthm} this can not happen for some maximal ideals because ZFC proves $\cov(\Mcal) \leq \non(\SNcal)$ (Miller~\cite{Mi}).


\bibliography{biblio}

\begin{thebibliography}{GKMS22}

\bibitem[Bar84]{BA84}
Tomek Bartoszy\'{n}ski.
\newblock Additivity of measure implies additivity of category.
\newblock {\em Trans. Amer. Math. Soc.}, 281(1):209--213, 1984.

\bibitem[Bar10]{BartInv}
Tomek Bartoszynski.
\newblock Invariants of measure and category.
\newblock In {\em Handbook of set theory. {V}ols. 1, 2, 3}, pages 491--555.
  Springer, Dordrecht, 2010.

\bibitem[BD{\v{S}}17]{BDS}
Lev Bukovsk\'{y}, Pratulananda Das, and Jaroslav {\v{S}}upina.
\newblock Ideal quasi-normal convergence and related notions.
\newblock {\em Colloq. Math.}, 146(2):265--281, 2017.

\bibitem[BHH04]{BHH}
B.~Balcar, F.~Hern\'{a}ndez-Hern\'{a}ndez, and M.~Hru\v{s}\'{a}k.
\newblock Combinatorics of dense subsets of the rationals.
\newblock {\em Fund. Math.}, 183(1):59--80, 2004.

\bibitem[BJ95]{BJ}
Tomek Bartoszy\'{n}ski and Haim Judah.
\newblock {\em Set theory. On the structure of the real line}.
\newblock A K Peters, Ltd., Wellesley, MA, 1995.

\bibitem[Bla86]{blass86}
Andreas Blass.
\newblock Near coherence of filters, {I}: {C}ofinal equivalence of models of
  arithmetic.
\newblock {\em Notre Dame J. Formal Logic}, 27(4):579--591, 1986.

\bibitem[Bla89]{BlassSP}
Andreas Blass.
\newblock Applications of superperfect forcing and its relatives.
\newblock In {\em Set theory and its applications ({T}oronto, {ON}, 1987)},
  volume 1401 of {\em Lecture Notes in Math.}, pages 18--40. Springer, Berlin,
  1989.

\bibitem[Bla10]{blassbook}
Andreas Blass.
\newblock Combinatorial cardinal characteristics of the continuum.
\newblock In {\em Handbook of set theory. {V}ols. 1, 2, 3}, pages 395--489.
  Springer, Dordrecht, 2010.

\bibitem[BF12]{BNF}
Piotr Borodulin-Nadzieja and Barnab\'{a}s Farkas.
\newblock Cardinal coefficients associated to certain orders on ideals.
\newblock {\em Arch. Math. Logic}, 51(1-2):187--202, 2012.

\bibitem[BM99]{BM}
Andreas Blass and Heike Mildenberger.
\newblock On the cofinality of ultrapowers.
\newblock {\em J. Symbolic Logic}, 64(2):727--736, 1999.

\bibitem[Bre91]{Br}
J{\"o}rg Brendle.
\newblock Larger cardinals in {C}icho\'n's diagram.
\newblock {\em J. Symbolic Logic}, 56(3):795--810, 1991.

\bibitem[Bre99]{brendle99}
J\"{o}rg Brendle.
\newblock Between {$P$}-points and nowhere dense ultrafilters.
\newblock {\em Israel J. Math.}, 113:205--230, 1999.

\bibitem[Bre02]{brendle02}
J{\"o}rg Brendle.
\newblock Mad families and iteration theory.
\newblock In {\em Logic and algebra}, volume 302 of {\em Contemp. Math.}, pages
  1--31. Amer. Math. Soc., Providence, RI, 2002.

\bibitem[BS87]{blassshelah}
Andreas Blass and Saharon Shelah.
\newblock There may be simple {$P_{\aleph_1}$}- and {$P_{\aleph_2}$}-points and
  the {R}udin-{K}eisler ordering may be downward directed.
\newblock {\em Ann. Pure Appl. Logic}, 33(3):213--243, 1987.

\bibitem[BS89]{blassshelah3}
Andreas Blass and Saharon Shelah.
\newblock Near coherence of filters {III}: {A} simplified consistency proof.
\newblock {\em Notre Dame J. Formal Logic}, 30(4):530--538, 1989.

\bibitem[BS92]{BartSh}
Tomek Bartoszy\'{n}ski and Saharon Shelah.
\newblock Closed measure zero sets.
\newblock {\em Ann. Pure Appl. Logic}, 58(2):93--110, 1992.

\bibitem[BS95]{BSch}
Tomek Bartoszy\'{n}ski and Marion Scheepers.
\newblock Filters and games.
\newblock {\em Proc. Amer. Math. Soc.}, 123(8):2529--2534, 1995.

\bibitem[Can88]{Canjar}
Michael Canjar.
\newblock Countable ultraproducts without {CH}.
\newblock {\em Ann. Pure Appl. Logic}, 37(1):1--79, 1988.

\bibitem[Car23]{Cfriendly}
Miguel~A. Cardona.
\newblock A friendly iteration forcing that the four cardinal characteristics
  of {$\mathcal{E}$} can be pairwise different.
\newblock {\em Colloq. Math.}, 173(1):123--157, 2023.

\bibitem[Eis01]{Eis01}
Todd Eisworth.
\newblock Near coherence and filter games.
\newblock {\em Arch. Math. Logic}, 40(3):235--242, 2001.

\bibitem[FS09]{FaSo}
Barnab\'{a}s Farkas and Lajos Soukup.
\newblock More on cardinal invariants of analytic {$P$}-ideals.
\newblock {\em Comment. Math. Univ. Carolin.}, 50(2):281--295, 2009.

\bibitem[GKMS22]{collapsing}
Martin Goldstern, Jakob Kellner, Diego~A. Mej\'{\i}a, and Saharon Shelah.
\newblock Controlling classical cardinal characteristics while collapsing
  cardinals.
\newblock {\em Colloq. Math.}, 170(1):115--144, 2022.

\bibitem[Hec72]{Hechlermad}
Stephen~H. Hechler.
\newblock Short complete nested sequences in {$\beta N\backslash N$} and small
  maximal almost-disjoint families.
\newblock {\em General Topology and Appl.}, 2:139--149, 1972.

\bibitem[Hru11]{Hrusak}
Michael Hru\v{s}\'{a}k.
\newblock Combinatorics of filters and ideals.
\newblock In {\em Set theory and its applications}, volume 533 of {\em Contemp.
  Math.}, pages 29--69. Amer. Math. Soc., Providence, RI, 2011.

\bibitem[HST22]{Bakke}
Karen~Bakke Haga, David Schrittesser, and Asger T\"{o}rnquist.
\newblock Maximal almost disjoint families, determinacy, and forcing.
\newblock {\em J. Math. Log.}, 22(1):Paper No. 2150026, 42, 2022.

\bibitem[Kat68]{katetov}
Miroslav Kat\v{e}tov.
\newblock Products of filters.
\newblock {\em Comment. Math. Univ. Carolinae}, 9:173--189, 1968.

\bibitem[Kec95]{kechris}
Alexander~S. Kechris.
\newblock {\em Classical Descriptive Set Theory}, volume 156 of {\em Graduate
  Texts in Mathematics}.
\newblock Springer-Verlag, New York, 1995.

\bibitem[K{\v{S}}W01]{I-conv}
Pavel Kostyrko, Tibor {\v{S}}al\'{a}t, and W\l adys\l~aw Wilczy\'{n}ski.
\newblock {$I$}-convergence.
\newblock {\em Real Anal. Exchange}, 26(2):669--685, 2000/01.

\bibitem[Laf96]{Laf}
Claude Laflamme.
\newblock Filter games and combinatorial properties of strategies.
\newblock In {\em Set theory ({B}oise, {ID}, 1992--1994)}, volume 192 of {\em
  Contemp. Math.}, pages 51--67. Amer. Math. Soc., Providence, RI, 1996.

\bibitem[LL02]{Laflear}
Claude Laflamme and Christopher~C. Leary.
\newblock Filter games on {$\omega$} and the dual ideal.
\newblock {\em Fund. Math.}, 173(2):159--173, 2002.

\bibitem[LZ98]{LZ}
Claude Laflamme and Jian-Ping Zhu.
\newblock The {R}udin-{B}lass ordering of ultrafilters.
\newblock {\em J. Symbolic Logic}, 63(2):584--592, 1998.

\bibitem[Mej13a]{Mmatrix}
Diego~Alejandro Mej{\'{\i}}a.
\newblock Matrix iterations and {C}ichon's diagram.
\newblock {\em Arch. Math. Logic}, 52(3-4):261--278, 2013.

\bibitem[Mej13b]{mejia2}
Diego~Alejandro Mej{\'{\i}}a.
\newblock Models of some cardinal invariants with large continuum.
\newblock {\em Ky\={o}to Daigaku S\=urikaiseki Kenky\=usho K\=oky\=uroku},
  1851:36--48, 2013.

\bibitem[Mil81]{Mi}
Arnold~W. Miller.
\newblock Some properties of measure and category.
\newblock {\em Trans. Amer. Math. Soc.}, 266(1):93--114, 1981.

\bibitem[Paw90]{P90}
Janusz Pawlikowski.
\newblock Finite support iteration and strong measure zero sets.
\newblock {\em J. Symbolic Logic}, 55(2):674--677, 1990.

\bibitem[Rep21a]{Miro1}
Miroslav Repick\'{y}.
\newblock Spaces not distinguishing ideal convergences of real-valued
  functions.
\newblock {\em Real Anal. Exchange}, 46(2):367--394, 2021.

\bibitem[Rep21b]{Miro2}
Miroslav Repick\'{y}.
\newblock Spaces not distinguishing ideal convergences of real-valued
  functions, {II}.
\newblock {\em Real Anal. Exchange}, 46(2):395--421, 2021.

\bibitem[Rot38]{Roth38}
Fritz Rothberger.
\newblock Eine \"{A}quivalenz zwishen der {K}ontinuumhypothese under der
  {E}xistenz der {L}usinschen und {S}ierpin-schishen mengen.
\newblock {\em Fund. Math}, 30:215--217, 1938.

\bibitem[RS23]{RaSt}
Dilip Raghavan and Juris Stepr\={a}ns.
\newblock The almost disjointness invariant for products of ideals.
\newblock {\em Topology Appl.}, 323:Paper No. 108295, 11, 2023.

\bibitem[\v{S}23]{SJpseudo}
Jaroslav \v{S}upina.
\newblock Pseudointersection numbers, ideal slaloms, topological spaces, and
  cardinal inequalities.
\newblock {\em Arch. Math. Logic}, 62(1-2):87--112, 2023.

\bibitem[{\v S}ot20]{Vthesis}
Viera {\v S}ottov\'{a}.
\newblock {\em The role of ideals in topological selection principles}.
\newblock Pavol Jozef \v{S}af\'arik University in Ko\v{s}ice, Faculty of
  Science, 2020.
\newblock (\v Sottov\'a is the first author's maiden name).

\bibitem[{\v S}{\v S}19]{SS}
Viera {\v S}ottov\'{a} and Jaroslav {\v S}upina.
\newblock Principle {${\rm S}_1(\mathcal{P},\mathcal{R})$}: ideals and
  functions.
\newblock {\em Topology Appl.}, 258:282--304, 2019.

\bibitem[{\v S}up16]{SJ}
Jaroslav {\v S}upina.
\newblock Ideal {QN}-spaces.
\newblock {\em J. Math. Anal. Appl.}, 435(1):477--491, 2016.

\bibitem[Tal80]{Talagrand}
Michel Talagrand.
\newblock Compacts de fonctions mesurables et filtres non mesurables.
\newblock {\em Studia Math.}, 67(1):13--43, 1980.

\bibitem[Voj93]{Vojtas}
Peter Vojt\'{a}\v{s}.
\newblock Generalized {G}alois-{T}ukey-connections between explicit relations
  on classical objects of real analysis.
\newblock In {\em Set theory of the reals ({R}amat {G}an, 1991)}, volume~6 of
  {\em Israel Math. Conf. Proc.}, pages 619--643. Bar-Ilan Univ., Ramat Gan,
  1993.



\end{thebibliography}
\bibliographystyle{alpha}

\end{document}

\begin{figure}[h]
\centering
\begin{tikzpicture}[scale=1]
\small{
\node (x) at (-2,0) {$\aleph_1$};
\node (a) at (0, 0) {$\add(\mathcal{N})$};
\node (b) at (4, 0) {$\add(\NstaridealJ{\Jcal})$};
\node (bbb) at (2, -1.5) {$\add(\NidealJ{\Jcal})$};
\node (c) at (6, 0) {$\non(\NstaridealJ{\Jcal})$};
\node (d) at (8, 0) {$\non(\NidealJ{\Jcal})$};
\node (aaa) at (6, -1.5) {$\mathfrak{s}$};
\node (e) at (10, 0) {$\non(\mathcal{N})$};
\node (aa) at (0, 2) {$\cov(\mathcal{N})$};
\node (bb) at (2, 2) {$\cov(\NidealJ{\Jcal})$};
\node (cc) at (4, 2) {$\cov(\NstaridealJ{\Jcal})$};
\node (ccc) at (4, 3.5) {$\mathfrak{r}$};
\node (dd) at (6, 2) {$\cof(\NstaridealJ{\Jcal})$};
\node (ddd) at (8, 3.5) {$\cof(\NidealJ{\Jcal})$};
\node (ee) at (10, 2) {$\cof(\mathcal{N})$};
\node (g) at (12,2) {$\mathfrak{c}$};
}

\foreach \from/\to in {x/a,a/b,b/c,c/d,d/e,aa/bb,bb/cc,cc/dd,ee/g,
a/aa,a/bbb,bb/ddd,b/cc,c/dd,e/ee,d/ddd,ddd/ee, aaa/c}
\draw [->] (\from) -- (\to);
\foreach \from/\to in {bbb/bb,bbb/d, dd/ee,cc/ccc} \draw [line width=.15cm,
white] (\from) -- (\to);
\foreach \from/\to in {bbb/bb,bbb/d, dd/ee,cc/ccc} \draw [->] (\from) -- (\to);
\end{tikzpicture}
\end{figure}

\vs{
I guess we can add the implications from $\add(\mathcal{N}_\Jcal)$ to $\add(\mathcal{N}^*_\Jcal)$ as well as from $\cov(\mathcal{N}^*_\Jcal)$ to $\cov(\mathcal{N}_\Jcal)$, but not seen the proof or argument right now.} \dm{Good question. The diagram would be nice if this is true, but I can't see an argument yet.} \vs{for me this is a pretty nice picture :D :D :D} \dm{Indeed, the figure is beautiful.}

Note that if we consider an~ideal $\Jcal$ on~$\omega$ which has the~Baire property then by \autoref{BaireForNJ} the~mentioned cardinal invariants of $\mathcal{N}_\Jcal$ equal to original ones.

\begin{figure}[H]
\centering
\setlength{\unitlength}{1.5mm}
\begin{picture}(76,25)
\put(23,1.5){\makebox(0,0){\footnotesize ?}}
\put(24,21.5){\makebox(0,0){\footnotesize ?}}
\put(50,1.5){\makebox(0,0){\footnotesize ?}}
\put(53,21.5){\makebox(0,0){\footnotesize ?}}
\put(3,0){\makebox(0,0){$\aleph_1$}}
\put(17,0){\makebox(0,0){\footnotesize add($\mnula_\Jcal$)}}
\put(30.8,0){\makebox(0,0){\footnotesize $\min\{\be_\Jcal,\mathfrak{d}^*_\Jcal\}$}}
\put(45,0){\makebox(0,0){ $\mathfrak{d}^*_\Jcal$}}
\put(59,0){\makebox(0,0){\footnotesize non($\mnula_\Jcal$)}}
\put(31,10){\makebox(0,0){$\be_\Jcal$}}
\put(45,10){\makebox(0,0){$\de_\Jcal$}}
\put(17,20){\makebox(0,0){\footnotesize cov($\mnula_\Jcal$)}}
\put(30.7,20){\makebox(0,0){$\mathfrak{b}^*_\Jcal$}}
\put(45,20){\makebox(0,0){\footnotesize $\max\{\mathfrak{b}^*_\Jcal,\de_\Jcal\}$}}
\put(59,20){\makebox(0,0){\footnotesize cof($\mnula_\Jcal$)}}
\put(73,20){\makebox(0,0){$\con$}}
\put(5,0){\vector(1,0){7.5}}
\put(21,0){\vector(1,0){4}}
\put(37,0){\vector(1,0){5.5}}
\put(48,0){\vector(1,0){6.5}}
\put(22,20){\vector(1,0){6.5}}
\put(33,20){\vector(1,0){5.5}}
\put(51,20){\vector(1,0){4}}
\put(63,20){\vector(1,0){8.5}}
\put(32.5,10){\vector(1,0){11}}
\put(17,2){\vector(0,1){16}}
\put(59,2){\vector(0,1){16}}
\put(31,2){\vector(0,1){6}}
\put(45,2){\vector(0,1){6}}
\put(31,12){\vector(0,1){6}}
\put(45,12){\vector(0,1){6}}
\end{picture}
\end{figure}